\documentclass[12pt]{iopart}

\usepackage{iopams}  
\usepackage{amssymb,amsfonts}
\usepackage{amsthm}
\usepackage{tikz}
\usetikzlibrary{trees}
\usepackage{forest}
\usepackage{qtree} 

\eqnobysec

\newtheorem{thm}{Theorem}[section]
\newtheorem{cor}[thm]{Corollary}
\newtheorem{prop}[thm]{Proposition}
\newtheorem{lemma}[thm]{Lemma}

\newtheorem{defn}[thm]{Definition}

\newtheorem{exmp}[thm]{Example}

\newtheorem{rem}[thm]{Remark}

\newcommand{\T}{\mathbb{T}}
\newcommand{\eqd}{\stackrel{d}{=}}
\newcommand{\roott}{{\bf 0}}

\begin{document}

\title[]{Directed polymers on a disordered tree with a defect subtree}

\author{Neal Madras\footnote{Corresponding author}}
\address{Department of Mathematics and Statistics, York University\\ 4700 Keele Street, Toronto, Ontario M3J 1P3, Canada}
\ead{madras@mathstat.yorku.ca}

\author{G\"{o}khan Y\i ld\i r\i m}
\address{Department of Mathematics, Bilkent University\\06800 Ankara, Turkey}
\ead{gokhan.yildirim@bilkent.edu.tr}

\begin{abstract}
We study the question of how the competition between \textit{bulk disorder} and a \textit{localized microscopic defect} affects the macroscopic behavior of a system in the directed polymer context at the free energy level. We consider the directed polymer model on a disordered $d$-ary tree and 
represent the localized microscopic defect by modifying the disorder distribution at each vertex in a
single path (branch), or  in a subtree, of the tree.
The polymer must choose between
following the microscopic defect and 
finding the best branches through the bulk disorder. We 
describe 
three possible phases, called the \textit{fully pinned, partially pinned} and \textit{depinned} phases. 
When the microscopic defect is 
associated only with a single branch, we compute the free energy and the critical curve of the model, 
and show that the partially pinned phase 
does not occur.
When the localized microscopic defect is
associated with a non-disordered 
regular 
subtree of the disordered tree, the picture is more complicated. 
We prove that all three phases are non-empty below a critical temperature, and that the partially 
pinned phase disappears above the critical temperature. 
\end{abstract}

\vspace{2pc}
\maketitle

\section{Introduction}

Directed polymers in a random environment are typical examples of models used to study the behavior of a one-dimensional object interacting with a disordered environment. In the mathematical formulation of these models, paths of a directed walk on a regular lattice or tree represent the directed polymer while an independent and identically distributed (i.i.d.)\ collection of random variables attached to the vertices
of the lattice/tree correspond to the random environment (\textit{bulk disorder}). 
Each path is assigned a Gibbs weight corresponding to the sum of the random variables of the visited vertices.
The polymer's interaction with the random environment is controlled by a parameter, $\beta$, which represents the \textit{inverse temperature}. The main questions are whether there exist different phases in the model depending on the temperature which manifest the effect of the disorder on the large scale behavior (diffusive versus superdiffusive) of the polymer, and how the phases can be characterized \cite{Comets1}. 
The earliest example of the model studied in the physics literature \cite{HH} 
(and then rigorously in \cite{IS}) was a $1+1$ dimensional lattice case as a model for the interface in 
a two-dimensional Ising model with random exchange interaction. Since then it has been used in models of various growth phenomena: formation of magnetic domains in spin-glasses \cite{HH}, vortex lines in superconductors \cite{Nel}, roughness of crack interfaces \cite{HHR}, and the KPZ equation \cite{KPZ}. The last twenty years have witnessed many significant results on the problems related to directed polymer models and more general 
polymer models. For a comprehensive introduction and an up-to-date summary of the results and methods for both the lattice and tree version of the directed polymer model, see the lecture notes \cite{Comets1}. For more general polymer models, see \cite{denH, Giacomin, Giacomin1}. 

A different direction of research considers polymers in a deterministic 
environment with a \textit{localized microscopic defect}.  
A primary example is the case of an interfacial layer between two fluids, 
modelled by a plane in a $3$-dimensional lattice (called a ``defect plane'' in some contexts),
such that each monomer of a polymer is energetically rewarded if it lies in this layer.
A related model is the situation that the monomers are attracted to an impenetrable 
wall of a container;
in this case, the polymer lives in a half-space bounded by an attracting plane.
There is typically a critical value of $\beta$ above which a positive fraction 
of the polymer is \textit{pinned} or \textit{adsorbed} to the surface, and below 
which the polymer is mostly free of the surface---that is, it is \textit{depinned} or \textit{desorbed}.
It is generally expected that the critical $\beta$ is strictly positive for an impenetrable 
boundary, and equal to zero for a penetrable boundary.  This problem can be solved 
exactly for directed polymers \cite{Giacomin1,Hamm1,Rubin}.  For the self-avoiding
walk model of polymers, the impenetrable result has been proven \cite{Hamm2}, while the 
penetrable case remains open, but can be proved under an extremely weak hypothesis
\cite{Madras}. 
In the very special case of self-avoiding walks at an impenetrable boundary on the honeycomb lattice, 
the exact critical value has been determined in \cite{BBDD}. 
Pinning problems also arise elsewhere, notably the context of high-temperature 
superconductors \cite{BSL, CMWT}.

\subsection{The bulk disorder versus a localized microscopic defect.}
\label{bulkvsmd}

In this paper, we shall study the question of how the competition between \textit{bulk disorder} and a \textit{localized microscopic defect} affects the macroscopic behavior of a system as reflected in pinning phenomena of directed polymers.
In the \textit{directed polymer on a disordered tree} model,  we add a fixed potential $u$ to each vertex 
on a \textit{branch} or a \textit{subtree} of the tree which represents the localized microscopic defect. 
Roughly speaking, the polymer must choose between following the localized microscopic defect and finding the best branch(es) through the bulk disorder. We see that there are three possible phases depending on the defect structure (a single branch versus a subtree) and the model parameters ($\beta, u$): 
\begin{itemize}
\item[-] \textbf{Fully pinned phase}, ${\cal R}_{FP}$: the partition function is dominated by polymer configurations that spend almost all their time in the defect structure. 
\item[-] \textbf{Depinned phase}, ${\cal R}_{D}$: the partition function is dominated by polymer configurations
that spend hardly any time in the defect structure.
\item[-] \textbf{Partially pinned phase}, ${\cal R}_{PP}$: the partition function is dominated by polymer configurations that spend a positive fraction (but not close to all) of their time inside the defect structure.
\end{itemize}
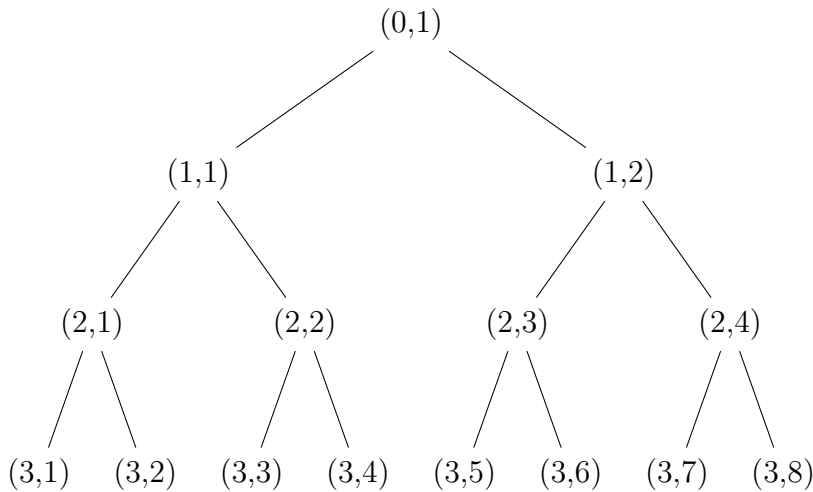
\begin{figure}
\begin{center}
\begin{forest}
  for tree={l+=0.8cm} 
  [{(0,1)}
    [{(1,1)}[{(2,1)}[{(3,1)}][{(3,2)}]][{(2,2)}[{(3,3)}][{(3,4)}]]]
    [{(1,2)}[{(2,3)}[{(3,5)}][{(3,6)}]][{(2,4)}[{(3,7)}][{(3,8)}]]]
  ]
\end{forest}
\caption{The nodes of a $d$-ary tree $\T$ are labeled by two integers $(k,j)$ where $k$ corresponds to the generation and $j$
enumerates the nodes within the $k^{th}$ generation from left to right. The root is labeled as $\roott=(0,1)$.}
\label{tree}
\end{center}
\end{figure}
For the formal definition of each phase, see Definition~\ref{phase-defn} in section~\ref{subtreesecproof}.

In the (nonrigorous) physics literature, this problem has been studied extensively in the lattice version of the directed polymer model \cite{BK, HN, K2, TL} but there have been disagreeing predictions for the $1+1$ dimensional lattice version as to whether the polymer follows the defect line as soon as the potential level $u$ is greater than $0$; for more details see section~\ref{latticemodel}. 
For some rigorous partial results in this direction, see \cite{AlYil, BSSV, BSV}. We consider this problem in the tree version of the directed polymer model which can be viewed as a mean field approximation of the lattice case.

In order to study our problem precisely, we first present some definitions and introduce some notation related to the \textit{directed polymers on disordered trees}, a model introduced in \cite{DerSp}. Let $\T$ be a rooted $d$-ary tree, in which every node  has exactly $d$ offspring ($d\geq 2$). We label the nodes of $\T$ by two integers $(k,j)$ where $k$ corresponds to the generation and $j\in \{1,2,\cdots,d^k\}$ numbers the nodes within the $k^{th}$ generation from left to right. The root is labeled as $\roott=(0,1)$. See Figure~\ref{tree}. An infinite directed path from the root is called a \textit{branch} of the tree.

We assume that every node $x=(k,j)$ of the tree $\T$ has an associated random variable denoted by $V(x)$ or $V_{k,j}$ that represents the random disorder at that node, all independent. 

The Hamiltonian of the model is defined as
\[    V\langle W\rangle   \;:=\;  \sum_{y\in W\setminus \{\roott\}} V(y)  \,   \]
where $W$ is a directed path in $\T$ from the root $\roott$ to some node in the $n^{th}$ generation. 

In the \textit{homogeneous disorder} (HD) model,
all the random variables $V(x)$ have the same distribution as some 
non-degenerate random variable $V$ with 
\begin{equation}
\label{moment}
\lambda(\beta):=\log \mathbf{E}[e^{\beta V}] <\infty \hbox{ for all  } \beta \in \mathbb{R}.
\end{equation}

The partition function of the HD model is defined as
\begin{equation}
\label{hompartfunction}
   Z_n^{HD}(\beta)   \;:=\;   \sum_{W:\roott\rightarrow (n,\cdot)}   e^{\beta V\langle W\rangle}  
\end{equation}
where the sum is over all directed paths $W$ in $\T$ from the root $\roott$ to some node in the 
$n^{th}$ generation. The parameter $\beta$ represents the \textit{inverse temperature}.

The free energy of the HD model is defined to be 
\begin{equation}
\label{phi}
\phi(\beta):=\lim_{n\to \infty}\frac{1}{n}\log Z^{HD}_n(\beta).
\end{equation}
%
For each $\beta$, this limit exists and is constant almost surely.   It is computed explicitly in \cite{BPP} as
\begin{equation}
\label{qfree}
\phi(\beta)=\left\{\begin{array}{ll}
\lambda(\beta) +\log d  &\hbox{ if } \quad \beta <\beta_c\\
\frac{\beta}{\beta_c}(\lambda(\beta_c)+\log d)  &\hbox{ if } \quad \beta \geq \beta_c \,
\end{array}\right.
\end{equation}
where 
the critical inverse temperature $\beta_c$ is the positive root of 
$\lambda(\beta)+\log d \,=\, \beta \lambda'(\beta)$, or is $+\infty$ if there is no root; 
see Lemma~\ref{criticalbeta} in section~\ref{homodisorder} for details.

The defect structure is incorporated into the model by assigning random variables from a different distribution
to the vertices in a part of the tree. 
Let $\tilde{\T}$ be the ``left-most'' $d_1$-regular subtree of the $d$-regular tree $\T$, with the same 
 root $\roott$ (for the precise definition see the beginning of section~\ref{mainresults}). 
 We assume that there are two possible distributions for $V(x)$, 
which we shall call $V$ and $\tilde{V}$:
\begin{verse}
  If $x\in \tilde{\T}$, then $V(x)$ has distribution $\tilde{V}$.   \\
  If $x\in \T\setminus \tilde{\T}$, then $V(x)$ has distribution $V$.
\end{verse}
We assume that $V$ satisfies Equation (\ref{moment}).  We shall consider two special cases:
\begin{verse}
  \textbf{Case I (Shift defect):} There is a real constant $u$ such that the \
  distribution of $\tilde{V}$ is  $V+u$. 
    \\
  \textbf{Case II (Nonrandom defect):} There is a real constant $u$ such 
  that $\mathbf{P}(\tilde{V}=u)\,=\,1$. 
\end{verse}

\subsubsection{Polymers on non-disordered trees with a defect subtree.}
\label{non-disorderedtrsec-int}
As a first case, we shall consider a directed polymer model on a \textit{deterministic} $d$-regular tree $\T$, no bulk disorder, and identify the localized microscopic defect with a $d_1$-regular subtree $\tilde \T$ of $\T$ by 
placing a fixed potential $u$ at each vertex of $\tilde \T$ and potential 0 elsewhere in $\T$. 
That is, we have $\mathbf{P}(V=0)\,=\,1$ and $\mathbf{P}(\tilde{V}=u)\,=\,1$.

We define the free  energy as
\begin{equation}
\label{free-det}
f^{Det}(\beta,u):= \lim_{n\rightarrow\infty}\frac{1}{n}\log Z_n^{Det}(\beta,u)
\end{equation}
where $Z^{Det}_n(\beta,u)$ is the partition function of the model. Note that $f^{Det}(\beta,0) \,=\, \log d$.

The critical curve is defined as
\begin{equation}
\label{critical-det}
u^{Det}_c(\beta):=\inf\{u\in \mathbb{R}: f^{Det}(\beta,u)>\log d\}.
\end{equation}
The following result is straightforward to prove (see Section \ref{non-disorderedtrsec}):
\begin{thm}
\label{mainthm-det} For any $\beta\geq 0$ and $u\in \mathbb{R}$, we have
\begin{equation}  
\label{Detfree-int}
f^{Det}(\beta,u)\;=\;
  \max\{ \beta u+\log d_1,\log d\}
\end{equation}
and hence 
\begin{equation}
\label{Detcritic-int}
u^{Det}_c(\beta) \;=\;  \frac{\log (d/d_1)}{\beta}\,.
\end{equation}
\end{thm}
We interpret the critical curve as follows.  When $u<u_c(\beta)$, then $f^{Det}(\beta,u)\,=\log d$,
which shows that the free energy is dominated by walks that are entirely (except for the root)
outside of $\tilde{\T}$; there are $(d-d_1)d^{n-1}$ such walks, each with weight 1 in the 
partition function.  This corresponds to the \textit{desorbed} or \textit{depinned} phase. 
In contrast, when $u>u_c(\beta)$, then $f^{Det}(\beta,u)\,=\beta u+\log d_1$,
which shows that the free energy is dominated by walks that are entirely 
in $\tilde{\T}$; there are $d_1^n$ such walks, each with weight $e^{\beta nu}$.
This corresponds to the \textit{fully adsorbed} or \textit{fully pinned} phase.

\subsubsection{Polymers on disordered trees with a defect branch.}
\label{dfbranchsec}

Next, we shall consider a directed polymer model on a 
$d$-regular tree $\T$ with bulk disorder and a one-dimensional microscopic shift defect.
Specifically, we identify the defect with the leftmost branch $\tilde \T$ of the tree $\T$
by adding a fixed potential $u$ to each vertex of $\tilde \T$; 
that is, the distribution of $\tilde{V}$ is  $V+u$.
See Figure~\ref{dbrfigure}. 
Therefore, for a directed path $W$ from the root to some node in the $n^{th}$ generation, 
the Hamiltonian is 
\[    V\langle W\rangle   \;:=\;  \sum_{y\in \tilde{\T}\,\cap\, (W\setminus \{\roott\})} (V(y)+u)  +
\sum_{y\in (\T\setminus \tilde \T)\, \cap \,(W\setminus \{\roott\})} V(y) \,.    
\]
\begin{figure}
\centering
\begin{tikzpicture}
\tikzstyle{level 1}=[sibling distance=40mm]
\tikzstyle{level 2}=[sibling distance=15mm]
\tikzstyle{level 3}=[sibling distance=3mm]
\node {root}
child {child{child child{edge from parent[thin]} child{edge from parent[thin]}} child{edge from parent[thin] child child child} child{edge from parent[thin] child  child child} edge from parent[ultra thick]}
child {child{child child child} child{child child child} child{child child child}}
child {child{child child child} child{child child child} child{child child child}};
\end{tikzpicture}
\caption{The thick edges represent the defect branch $\tilde{\T}$ of the tree $\T$. 
We assume that $V(x)$ has distribution $V$ for each $x\in \T\setminus \tilde{\T}$,  whereas $V(x)$ has distribution $\tilde{V}=V+u$ for each $x\in \tilde{\T}$. 
When $u>u_c(\beta)$, the polymer follows the defect branch.} 
\label{dbrfigure}
\end{figure}
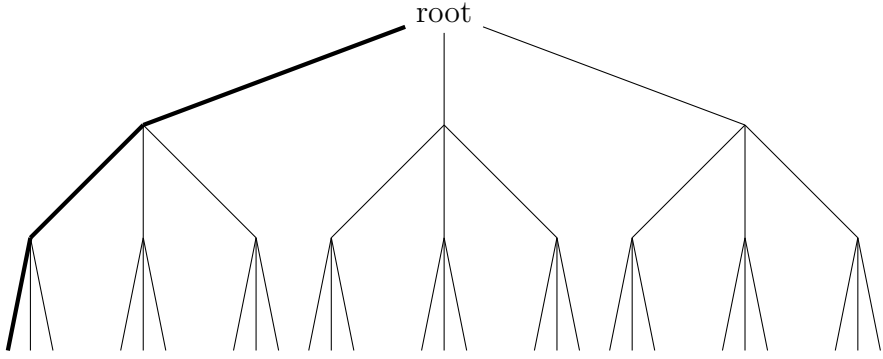

Then the free energy of the model is defined as 
\begin{equation}
\label{freeeng-int}
f^{Br}(\beta,u):=\lim_{n\to \infty}\frac{1}{n}\log Z^{Br}_n(\beta,u)
\end{equation}
where $Z^{Br}_n(\beta,u)$ is the partition function of the model, defined as in the right-hand side of
Equation (\ref{hompartfunction}).
The existence of the limit in Equation~(\ref{freeeng-int}) is part of the assertion of 
Theorem~\ref{mainthm} below.   Observe that $f^{Br}(\beta,0)\,=\,\phi(\beta)$ (recall Equation (\ref{phi})).

We define the critical curve as
\begin{equation}
u^{Br}_c(\beta):=\inf\{u\in \mathbb{R}: f^{Br}(\beta,u)>\phi(\beta)\}.
\end{equation}
In our next result, we compute the free energy and the critical curve explicitly. In the statement of the theorem, the quantity $\beta_c$ is the critical inverse temperature for the 
homogeneous disorder model, see Equation~(\ref{qfree}) and section~\ref{homodisorder}.
\begin{thm}
\label{mainthm} For any $\beta\geq 0$ and $u\in \mathbb{R}$, we have almost surely
\begin{equation}
f^{Br}(\beta,u)=\max\{\beta u+\beta\mu,\phi(\beta)\} 
\end{equation}
where $\mu=\mathbf{E}(V)$.
Hence
\begin{equation}
u^{Br}_c(\beta)=\left\{\begin{array}{ll}
\frac{1}{\beta}(\lambda(\beta) +\log d)-\mu &\hbox{ if }\quad \beta < \beta_c\\
\frac{1}{\beta_c}(\lambda(\beta_c)+\log d)-\mu &\hbox{ if }\quad \beta \geq \beta_c \,.
\end{array}\right.
\end{equation}
\end{thm}
We also see that for this model the partially pinned phase,  ${\cal R}_{PP}$, is always empty.   Indeed, we show in the proof of Theorem~\ref{mainthm} (see Section 
\ref{dfbranchsec}) that $\beta u+\beta\mu$ is the contribution to the free energy from the path $W$
that lies in $\tilde{\T}$.
\begin{rem} 
Note that for a non-degenerate random variable $V$,  $e^{\lambda(\beta)}=\mathbf{E}(e^{\beta V})>e^{\beta \mathbf{E}(V)}=e^{\beta \mu}$. Therefore
\begin{equation}
\label{Dbcriticlb}
u^{Br}_c(\beta)>\left\{
\begin{array}{ll}
\frac{\log d}{\beta} &\hbox{ if } \quad \beta < \beta_c\\
\frac{\log d}{\beta_c} &\hbox{ if } \quad \beta\geq  \beta_c.
\end{array}\right.
\end{equation}
From Equations~(\ref{Detcritic-int}) (with $d_1=1$) and (\ref{Dbcriticlb}), we  see that quenched 
randomness shifts the critical curve, that is, $u^{Br}_c(\beta)>u^{Det}_c(\beta)$ for all $\beta>0$. 
\end{rem}

\subsubsection{Polymers on disordered trees with a non-disordered defect subtree.}
\label{subtreesec}
 In this section, we shall consider a different microscopic defect structure that is identified with a \textit{deterministic} $d_1$-regular subtree $\tilde \T$ of the $d$-regular tree $\T$, and we identify the bulk disorder with the vertices in $\T\setminus \tilde \T$; that is, there is a real constant $u$ such that 
 $\tilde{V}\equiv u$.  See Figure~\ref{dstfigure}.
Therefore for a directed path $W$ from the root to some node in the $n^{th}$ generation, 
the Hamiltonian is 
\[V\langle W\rangle   \;:=\;  \sum_{y\in \tilde{\T}\cap (W\setminus \{\roott\})} \!\!\! u   \;\;+
\sum_{y\in (\T\setminus \tilde \T)\cap (W\setminus \{\roott\})} \!\!\! V(y) \,.  
\]
We denote the partition function of the model with a defect subtree by $Z^{ST}_n(\beta,u)$. 
For this model, it is not obvious how to prove that the limiting free energy exists almost surely; see the beginning of section~\ref{subtreesecproof}.

\begin{figure}
\centering
\begin{tikzpicture}
\tikzstyle{level 1}=[sibling distance=40mm]
\tikzstyle{level 2}=[sibling distance=15mm]
\tikzstyle{level 3}=[sibling distance=3mm]
\node {root}
child {child{child child child{edge from parent[thin]}} child{child child child{edge from parent[thin]}} child{edge from parent[thin] child  child child} edge from parent[ultra thick]}
child {child{child child child{edge from parent[thin]}} child{child child child{edge from parent[thin]}} child{edge from parent[thin] child  child child} edge from parent[ultra thick]}
child {child{child child child} child{child child child} child{child child child}};
\end{tikzpicture}
\caption{The thick edges represent the defect subtree $\tilde{\T}$ of the  tree $\T$.    
Here, $d=3$ and $d_1=2$.  The disorder
$V(x)$ has distribution $V$ for each $x\in \T\setminus \tilde{\T}$,  whereas $V(x)$ has distribution $\tilde{V}$
for each $x\in \tilde{\T}$.  In Section 1.1.3, we assume that $\tilde{V}$ is almost surely constant, that is, $\tilde{V}\equiv u$.} 
\label{dstfigure}
\end{figure}
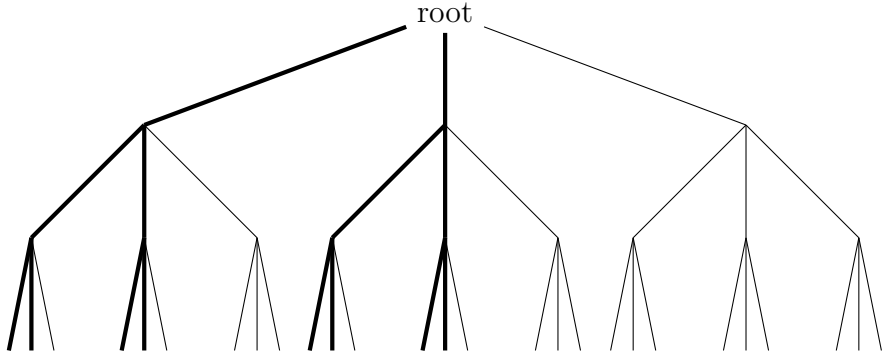

We shall define the following functions:  
 \begin{eqnarray}
      \label{functionF}
      F(\beta)   & := & \frac{1}{\beta}\,(\lambda(\beta)+\log d -\log d_1)   \\
      \label{functionJ}
      J(\beta) & := &   \frac{1}{\beta}\left( \phi(\beta)-\log d_1-
      \frac{[\lambda(\beta)+\log d-\phi(\beta)]\log d_1}{\lambda(2\beta)-2\lambda(\beta)-\log d}
         \right) .
 \end{eqnarray}

Our main result for this model is the following.   
The proof appears in section \ref{subtreesecproof}. See also Figure~\ref{phasediagramfig}. 
 
\begin{thm}
     \label{mainthST}
(a)  For every $\beta \in [0,\beta_c]$,  we have
\[     (\beta,u)    \;\in  \;  \left\{   \begin{array}{ll}
      {\cal R}_{FP}   & \hbox{ if }  \quad u\geq F(\beta)   \\
      {\cal R}_{D} & \hbox{ if } \quad u\leq F(\beta) \,.
         \end{array} \right.
\]
(b)
For every $\beta>\beta_c$, we have
\[     (\beta,u)    \;\in  \;  \left\{   \begin{array}{ll}
      {\cal R}_{FP}   & \hbox{ if } \quad  u\geq F(\beta)   \\
      {\cal R}_{D}  & \hbox{ if } \quad u\leq F(\beta_c)    \\
      {\cal R}_{PP}  & \hbox{ if } \quad J(\beta)<u<F(\beta)   \\
      {\cal R}_{D}\cup {\cal R_{PP}} & \hbox{ if } \quad F(\beta_c)<u\leq J(\beta).
      \end{array} \right.
\]
\end{thm}
\noindent
We also prove in Proposition \ref{prop-FJ} that $F(\beta_c)<J(\beta)<F(\beta)$ whenever $\beta>\beta_c$.
This shows that for every $\beta>\beta_c$, there is a value of $u$ such that $(\beta,u)\in {\cal R}_{PP}$.
That is, the partially pinned phase appears as soon as $\beta$ exceeds $\beta_c$.

\begin{figure}[t!]
\centering
\includegraphics[scale=0.9]{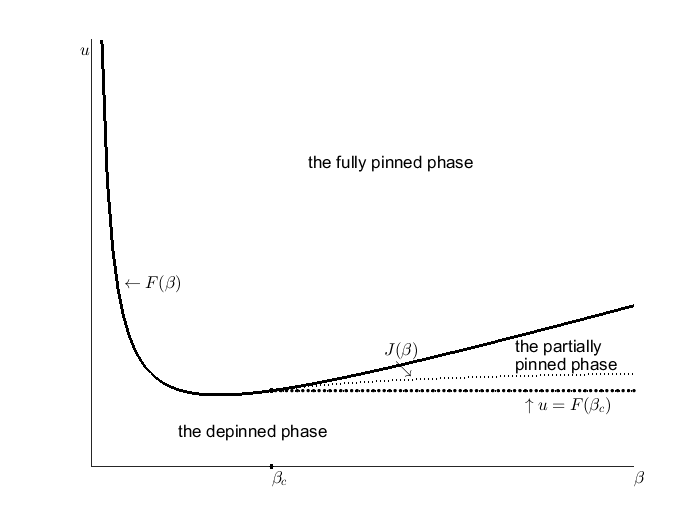}
\caption{Phase diagram for the model with a non-disordered defect subtree, from Theorem~1.4. 
The value $\beta_c$ is the critical inverse temperature for the phase transition between \textit{weak} and
\textit{strong} disorder in the 
\textit{homogeneous disorder} model; 
see section~1.2. 
For the \textit{fully pinned} phase, ${\cal R}_{FP}$, the dominant terms in the partition function are the walks that spend almost all their time in the defect subtree,  whereas for the \textit{depinned} phase, ${\cal R}_{D}$, the walks that spend hardly any time in the defect subtree dominate the partition function. In contrast, the 
dominant walks in the \textit{partially pinned} phase, ${\cal R}_{PP}$, are those for which the fraction
of time spent in the defect subtree is bounded away from 0 and from 1. The boundary curves $F$ and $J$ are given explicitly in Equations~(1.14) and (1.15). Our characterization of the phases is not complete for 
$\beta>\beta_c$ when $u$ is between $F(\beta_c)$ and $J(\beta)$. The point $(\beta_c,F(\beta_c))$ is the leftmost boundary point of the partially pinned phase, by 
Proposition~2.5.
}
\label{phasediagramfig}
\end{figure}

In a different direction of research \cite{BSS}, Basu, Sidoravicious and Sly considered the question of \textit{``how a localized microscopic defect, even if it is small with
respect to certain dynamic parameters, affects the macroscopic behavior of a system''} in the context of two classical exactly solvable models: Poissonized version of Ulam's problem of the maximal increasing
sequence and the totally asymmetric simple exclusion process. In the first model, by using a Poissonized version of directed last passage percolation on $\mathbb{R}^2$, they introduced the microscopic defect by adding a small positive density of extra points along the diagonal line. In the second, they introduced the microscopic defect by slightly decreasing the jump rate of each particle when it crosses the origin. They showed that in Ulam's problem the time constant increases, and for the exclusion process the flux of particles decreases. Thereby, they proved that in both cases the presence of an arbitrarily small microscopic defect affects the macroscopic behavior of the system, and hence settled the longstanding ``Slow Bond Problem'' from statistical physics.

The rest of the paper is organized as follows.   In section~\ref{homodisorder}, we introduce some notation, 
review the directed polymer on disordered tree model, and summarize the main existing results which we use in this paper. In section~\ref{mainresults}, we prove our results:  
Theorem \ref{mainthm-det} is proved in section \ref{non-disorderedtrsec},
Theorem \ref{mainthm} in section \ref{dfbranchsec}, 
and Theorem \ref{mainthST} in section \ref{subtreesecproof}.
 We conclude by discussing our results and some related models in section \ref{finalsec}.

For two random variables $X$ and $Y$, we use the notation 
$X\,\stackrel{d}{=}\,Y$ to denote that they have the same distribution. If a probability statement is true 
with probability one, then we use the phrase ``almost sure,'' abbreviated ``a.s.''.

\subsection{Polymers on trees with homogeneous disorder.}
\label{homodisorder}
In this section, we present some definitions and review the main results related to directed polymers on disordered trees. Let $\T$ be a rooted $d$-ary tree, in which every node  has exactly $d$ offspring ($d\geq 2$). Recall that we label the nodes of $\T$ by two integers $(k,j)$ where $k$ corresponds to the generation and $j\in \{1,2,\cdots,d^k\}$ numbers the nodes within the $k^{th}$ generation. 
The root is  $\roott=(0,1)$.  The set of offspring of node $(k,j)$ is 
$\{(k+1,(j-1)d+\ell) \,:\,1\leq \ell \leq d\}$. See Figure~\ref{tree}. If $x=(k,j)$, then we say that $k$ is the \textit{generation} or \textit{height} of $x$, and we write
$k= \textrm{Height}(x)$.
We assume that every node $x=(k,j)$ of the tree $\T$ has an associated random variable denoted by $V(x)$ or $V_{k,j}$ that represents the random disorder at that node, all independent and with the same distribution as $V$.

Define
\begin{equation}
\label{thatf}
f(\beta):=\lambda(\beta)+\log d-\beta \lambda'(\beta) \quad \hbox{for} \quad \beta\geq 0
\end{equation}
where $\lambda$ comes from Equation~(\ref{moment}).

Note that $\lambda$ is a strictly convex function of $\beta$, and therefore we have 
$f'(\beta)<0$ and $f(\beta)<\log d$ for all $\beta>0.$

For the proof of the following lemma, see \cite{Comets, MO}.
\begin{lemma}
\label{criticalbeta}
$f$ has a unique positive root if and only if either
\begin{itemize}
\item[-] $V$ is unbounded, or
\item[-] $w:=\hbox{ess}\sup V$ is finite and $\mathbf{P}(V=w)<1/d$.
\end{itemize}
We use $\beta_c$ to denote the unique positive root of $f$. If no solution exists, then $\beta_c=\infty$.
\end{lemma}

Recall that $Z^{HD}_n(\beta)$ denotes the \textit{homogeneous disorder} partition function defined in Equation~(\ref{hompartfunction}).

The following positive martingale $(M_n(\beta), \mathcal{F}_n)_{n\geq 0}$ has
played a  crucial role in the analysis of the model:
\begin{equation*}
M_n(\beta):=\frac{Z^{HD}_n(\beta)}{\mathbf{E} Z^{HD}_n(\beta)}
\end{equation*}
where $\mathcal{F}_n=\sigma\{V(x): \textrm{Height}(x)\leq n\}$
is the $\sigma$-algebra generated by all the random variables between generation 1 and $n$. The martingale methods are first used in \cite{Bolt} in the lattice version of the directed polymer model
 and then in \cite{BPP} for the tree version. From the Martingale Convergence Theorem, it follows that $ M(\beta):=\lim_{n\to \infty}M_n(\beta)$ exists almost surely and Kolmogorov's zero-one law implies that $\mathbf{P}(M(\beta)=0)\in \{0,1\}$ because $\{M(\beta)>0\}$ is a tail event. 
 It is known that \cite{Biggins, KahPey}
 \begin{eqnarray*}
M(\beta)>0&& \hbox{ almost surely for all } \quad 0\leq \beta <\beta_c\\ 
M(\beta)=0&& \hbox{ almost surely for all } \quad \beta \geq \beta_c
\end{eqnarray*} 
where $\beta_c$ comes from Lemma~\ref{criticalbeta}.
The first case is called the \textit{weak disorder regime} and the second case is called the \textit{strong disorder regime} \cite{Comets1}. 
Recall that the critical inverse temperature $\beta_c$ also marks a phase transition in the model at the level of the
free energy $\phi$ which is defined in Equation~(\ref{phi}).

The strong disordered regime can be considered as the \textit{energy dominated} or \textit{localized} phase as a single polymer configuration supports the full free energy whereas the weak disorder regime can be considered as the \textit{entropy dominated} or \textit{delocalized} phase as the full free energy is supported by a random sub-tree of positive exponential growth rate, which is strictly smaller than the growth rate of the full tree \cite{MO}. Note also that $M_n(\beta)$ converges to zero exponentially fast for $\beta>\beta_c$, but even though $\beta_c$ is in the strong disorder regime the decay rate of $M_n(\beta_c)$ is not exponential \cite{HuShi}. 

The following concentration result is proven in Proposition 2.5 of \cite{CSY1} for the partition function of the lattice version of the directed polymer model, and it is easy to see that it also holds true for the tree version of model.
\begin{prop}[\cite{CSY1}]
For any $\epsilon>0$ and $\beta\geq 0$, there exists $N:=N(\beta,\epsilon)$ such that
\begin{equation}
\label{concentration}
\fl \mathbf{P}(|\log Z^{HD}_n(\beta) -\mathbf{E}\log Z^{HD}_n(\beta)|\geq n\epsilon)\;\leq\; \exp{\left({-\frac{\epsilon^{2/3}n^{1/3}}{4}}\right)}, \qquad n\geq N \,.
\end{equation}
\end{prop}
By combining Equations (\ref{qfree}) and (\ref{concentration}), we also get
\begin{equation}
\label{meanlimit}
\phi(\beta)=\lim_{n\to \infty}\frac{1}{n}\mathbf{E}\log Z^{HD}_n(\beta).
\end{equation}
Observe that 
\[  \mathbf{E}Z^{HD}_n(\beta)\;=\; d^n(e^{\lambda(\beta)})^n\;=\; e^{n\lambda(\beta)+n\log d}  
\]
 and hence
\begin{equation}
   \label{eq.logElim}
    \lim_{n\to \infty}\frac{1}{n}\log \mathbf{E}Z^{HD}_n(\beta)=\lambda(\beta) +\log d.
\end{equation}
We also note that 
\begin{equation}
   \label{eq.phileqlam}
   \phi(\beta)\,\leq \,\lambda(\beta)+\log d  \hspace{5mm}\hbox{ for every $\beta$}
\end{equation}
(for example, by Equations (\ref{meanlimit}) and (\ref{eq.logElim}) and Jensen's inequality).  Indeed,
Equation (33) of \cite{BPP} tells us that
\begin{equation}
   \label{eq.philtlam}
   \phi(\beta)\,< \,\lambda(\beta)+\log d  \hspace{5mm}\hbox{ for every $\beta>\beta_c$.}
\end{equation}

\section{Proofs of the Main Results}
\label{mainresults}
Before we  prove our results, we introduce some more notation. 
We assume that $1\leq d_1<d$.
Let $\tilde{\T}$ be the ``left-most'' $d_1$-regular subtree of the $d$-regular tree $\T$, with the same 
 root $\roott$, where ``left-most'' is interpreted as follows.
For a node $x\in \tilde{\T}$, let $\tilde{D}(x)$ be the set of nodes in $\tilde{\T}$ whose parent is $x$,
and let $D(x)$ be the set of nodes in $\T\setminus \tilde{\T}$ whose parent is $x$. 
Using the notation $x=(k,j)$, we specify
\begin{eqnarray*}
     \tilde{D}(x)   & := & \{ (k+1,d(j-1)+\ell ) \,: \, 1\leq \ell \leq d_1\}    \\
     D(x)   & := & \{ (k+1,d(j-1)+\ell) \,: \, d_1<\ell \leq d\}  \,.
\end{eqnarray*}
For $d=3$, the cases $d_1=1$ and $d_1=2$ are depicted in  Figures \ref{dbrfigure} and \ref{dstfigure} 
respectively. 
Observe that
\[   |\tilde{D}(x)| \,=\, d_1   \hspace{5mm}\hbox{and} \hspace{5mm} |D(x)|\,=\, d-d_1  
   \hspace{5mm}  \hbox{for $x\in \tilde{\T}$}.
\]

Observe that for every directed path $W=(w(0),w(1),\ldots,w(n))$ with $w(0)=\roott$ and 
$\textrm{Height}(w(n))=n$, there is an integer $k\in [0,n]$ such that $w(m)\in \tilde{\T}$ if and
only if $m\leq k$.  That is, once the path leaves $\tilde{\T}$, it never returns to $\tilde{\T}$.
Many of our calculations involve summing over values of this quantity $k$.

Recall that we assume that there are two possible distributions for $V(x)$, 
called $V$ and $\tilde{V}$:
\begin{verse}
  If $x\in \tilde{\T}$, then $V(x)$ has distribution $\tilde{V}$.   \\
  If $x\in \T\setminus \tilde{\T}$, then $V(x)$ has distribution $V$.
\end{verse}

For a node $x$ with $\textrm{Height}(x)\leq n$, let 
\begin{equation}
\label{partfunction}
   Z_n^{[x]}(\beta)   \;:=\;   \sum_{W:x\rightarrow (n,\cdot)}   e^{\beta V\langle W\rangle} 
\end{equation}
where the sum is over all directed paths $W$ in $\T$ from $x$ to some node in the $n^{th}$ generation,
and the Hamiltonian is 
\[    V\langle W\rangle   \;:=\;  \sum_{y\in W\setminus \{x\}} V(y)  \,.    \]

We shall write
$Z_n^{[x]}(\beta)$ for whichever model is under
consideration, suppressing other details from the notation.

\bigskip

\subsection{The deterministic model:  Proof of Theorem~\ref{mainthm-det}.}
\label{non-disorderedtrsec}

Recall that we have $\mathbf{P}(V=0)\,=\,1$ and $\mathbf{P}(\tilde{V}=u)\,=\,1$ for this model.  
Then the partition function can be written explicitly as 
\[   Z^{Det}_n(\beta,u)\;=\;  \sum_{k=0}^{n-1}e^{k\beta u} d_1^k(d-d_1)  d^{n-k-1}+d^n_1e^{n\beta u }.
\]
The free energy, Equation~(\ref{free-det}), exists because 
\[   \max\{ e^{n\beta u}d_1^n,(d-d_1)d^{n-1} \}  \;\leq \; Z^{Det}_n(\beta,u)
 \; \leq \;  (n+1)\,(\max\{e^{\beta u}d_1,d\})^n
\]
which shows that
\begin{equation}  
\label{Detfree}
f^{Det}(\beta,u)\;=\;
  \max\{ \beta u+\log d_1,\log d\}.
\end{equation}
From Equation~(\ref{Detfree}), it follows that the critical curve defined in Equation~(\ref{critical-det})  is given by
\begin{equation}
\label{Detcritic}
u^{Det}_c(\beta) \;=\;  \frac{\log (d/d_1)}{\beta}\,.
\end{equation}

\bigskip

\subsection{The defect branch:  Proof of Theorem \ref{mainthm}.}
\label{dfbranchsec}
First, we shall introduce some notation. For each nonnegative integer $m$, let $S_m$ be the sum of the
(unshifted) disordered variables along the 
left-most branch of the tree up to the $m^{th}$ generation, that is,
\begin{equation*}
S_m \;:=\;\sum_{k=1}^m V_{k,1} \hspace{5mm}\hbox{ and } \hspace{5mm} S_0=0.
\end{equation*}
Recalling the definition in Equation (\ref{partfunction}), let 
\begin{eqnarray}
   \label{eq.GBR1}
G^{Br}_{k,n}(\beta)\; :=\;\sum _{y\in D((k,1))}e^{\beta V(y)}Z_n^{[y]}(\beta) \hspace{5mm}\hbox{ for }   
  0\leq k<n,
   \hbox{  and}   \\
   \nonumber
  G^{Br}_{n,n}(\beta) \;  :=\; 1. 
\end{eqnarray}
That is, $G^{Br}_{k,n}(\beta)$ is the sum of all contributions by walks from node $x=(k,1)$ up
to height $n$ that do not pass through the node $(k+1,1)$.

Then the partition function can be written as
\begin{equation}
   \label{eq.ZsumG1}
     Z^{Br}_n(\beta,u)  \;=\; \sum_{k=0}^n e^{\beta(uk+S_k)}G^{Br}_{k,n}(\beta).
\end{equation}

Observe that for each $y \in D((k,1))$, the quantities $Z_n^{[y]}(\beta)$ and $Z^{HD}_{n-k-1}(\beta)$ 
have the same distribution.  
Moreover, we see that the sum $G^{Br}_{k,n}(\beta)$ is \textit{stochastically smaller} than 
$Z^{HD}_{n-k}(\beta)$, which by definition means that 
\begin{eqnarray}
   \label{eq.stochP}
     \mathbf{P}(G^{Br}_{k,n}(\beta) \,\geq \,A)  \; \; & \leq \;\;  \mathbf{P}(Z^{HD}_{n-k}(\beta) \,\geq \,A)
     \hspace{4mm}\hbox{for every real constant $A$. }  
\end{eqnarray}

Fix $\beta$ and $u$, and 
let $\epsilon>0$ be given. By Equation (\ref{meanlimit}), there exists a constant 
$n_o=n_o(\epsilon,\beta)\geq 1$ such that 
\begin{equation}
\label{meanbound}
     \mathbf{E}\log Z_j^{HD}(\beta) \; \leq \; j(\phi(\beta)+\epsilon)
   \hspace{6mm}\hbox{for every integer $j\geq n_o$.}
\end{equation}
From Equations (\ref{eq.stochP}), (\ref{meanbound}), and (\ref{concentration}), there exist 
$n_1=n_1(\epsilon,\beta)$ and $c=c(\epsilon)$ such that for all 
nonnegative integers $k$ and $n$ with
$n-k\geq n_1$, we have
\begin{eqnarray}
\nonumber
\fl
\mathbf{P}\left(G_{k,n}^{Br}(\beta)\geq e^{(n-k)[\phi(\beta)+2\epsilon]}\right)\; & \leq \;
\mathbf{P}\left(Z_{n-k}^{HD}(\beta)\geq e^{(n-k)[\phi(\beta)+2\epsilon]}\right)    \\
   \nonumber
   & \leq \;  \mathbf{P}\left(\log Z_{n-k}^{HD}(\beta)\geq \mathbf{E}\log Z_{n-k}^{HD}(\beta)
     +(n-k)\epsilon \right)    \\
      &\leq \; e^{-c(n-k)^{1/3}}\,. 
    \label{conpart}
\end{eqnarray}

Let us define the quantities $W\,:=\,  \max\{\beta(u+\mu),\phi(\beta)\}$ and
\begin{equation}
   p_n\; :=\; \mathbf{P}\left( Z^{Br}_n(\beta,u)
    \;\geq \;  (n+1) e^{n(W+3\epsilon)}\right)   \hspace{5mm}\hbox{for $n\geq 1$.}
\end{equation}

By Equation (\ref{eq.ZsumG1}), for every $n$ we obtain 
\begin{eqnarray*}
p_n  &\leq & 
   \sum_{k=0}^{n}\mathbf{P}\left(e^{\beta(ku+S_k)}G_{k,n}^{Br}(\beta)\,\geq \,
    e^{k(W+\epsilon )+n\epsilon}e^{(n-k)(W+\epsilon)+n\epsilon}\right)\\
&\leq & \sum_{k=0}^{n}\mathbf{P}\left(e^{\beta(ku+S_k)}G_{k,n}^{Br}(\beta)\,\geq \,
    e^{k(\beta(u+\mu)+\epsilon )+n\epsilon}e^{(n-k)(\phi(\beta)+\epsilon)+n\epsilon}\right)\\
  & \leq & A_n+B_n \,,
\end{eqnarray*}
where
\begin{eqnarray*}  
   A_n \;&:=\;  \sum_{k=0}^{n}\mathbf{P}\left(e^{\beta(ku+S_k)}\,\geq \,
    e^{k(\beta(u+\mu)+\epsilon )+n\epsilon}\right)   
    \hspace{5mm}\hbox{and}  \\
    B_n\;&:=\; \sum_{k=0}^{n}\mathbf{P}\left(G_{k,n}^{Br}(\beta)\,\geq \,
    e^{(n-k)(\phi(\beta)+\epsilon)+n\epsilon}\right)\,.
\end{eqnarray*}

We shall handle $A_n$ by a standard ``large deviation'' bound.
Recall that $\lambda(\beta)=\log \mathbf{E}[e^{\beta V}]$ and $\mu=\mathbf{E}(V)$.
For  every $t>0$ and  every $\alpha>0$, we have
\begin{equation*}
   \mathbf{P}(S_m\geq m(\mu+t)) \; \leq\; e^{-m[\alpha(\mu+t)-\lambda(\alpha)]}
  \hspace{5mm}\hbox{for every $m\geq 1$}
\end{equation*}
(see for example Equation (2.6.2) of \cite{Durrett}).
Since $\lambda'(0)=\mu$, there exists $\alpha^*>0$ such that $\alpha^*(\mu+\epsilon)-\lambda(\alpha^*)>0$
(this is Lemma 2.6.2 of \cite{Durrett}).   Therefore, for every $k\in [1,n]$, we have
\begin{eqnarray*}
    \mathbf{P}\left(S_k \,\geq \, k \left(\mu+\epsilon+\frac{n\epsilon}{k\beta}\right)\right)
    \;\leq \; e^{-k[\alpha^*(\mu+\epsilon+\frac{n\epsilon}{k\beta})-\lambda(\alpha^*)]}
    \;<\; e^{-n(\alpha^*\epsilon/\beta)}.
\end{eqnarray*}
Thus, observing that the $k=0$ summand of $A_n$ equals 0, we have
\begin{equation}
   \label{eq.sumAf}
     A_n   \;=\;   \sum_{k=1}^{n}\mathbf{P}
      \left(S_k\geq k\left(\mu+\epsilon+\frac{n\epsilon}{k\beta}\right) \right)
         \;\leq \;  n \,e^{-n(\alpha^*\epsilon/\beta)}
         \hspace{3.4mm}\hbox{for all $n\geq 1$. }
\end{equation}
For $B_n$, let $m_n=\lfloor n-\sqrt n\rfloor$.  
Using Equations (\ref{eq.stochP}) and (\ref{conpart}) and Markov's inequality, 
we deduce that if $\sqrt{n}>n_1(\epsilon,\beta)$, then  
\begin{eqnarray}
   \nonumber
  B_n &\leq &
  \sum_{k=0}^{m_n}\mathbf{P}\left(G_{k,n}^{Br}(\beta)\geq e^{(n-k)(\phi(\beta)+\epsilon)+
    (n-k)\epsilon}\right)\\
    \nonumber
&& \quad + \;\sum_{k=m_n+1}^{n}\mathbf{P}\left(Z_{n-k}^{HD}(\beta)\geq e^{(n-k)(\phi(\beta)+\epsilon)
   +n\epsilon}\right)\\
   \nonumber
&\leq &\sum_{k=1}^{m_n} e^{-c(n-k)^{1/3}} + \sum_{k=m_n+1}^{n} 
  \frac{\mathbf{E} Z_{n-k}^{HD}(\beta) }{e^{(n-k)[\phi(\beta)+\epsilon]+n\epsilon}}\\
  \nonumber
&\leq & n e^{-c n^{1/6}}+\sum_{k=m_n+1}^{n} \frac{e^{(n-k)(\lambda(\beta)+\log d)}}{
     e^{(n-k)\phi(\beta)} \,e^{n\epsilon} }\\
     \label{eq.sumBf}
&\leq & n e^{-c n^{1/6}} +(\sqrt{n}+1)\, e^{-n\epsilon} e^{\sqrt n (\lambda(\beta)+\log d-\phi(\beta))}
\end{eqnarray}
(where the last inequality holds because $\phi(\beta)\leq \lambda(\beta)+\log d$).

Recalling that $p_n\,\leq \, A_n+B_n$, we see from Equations (\ref{eq.sumAf}) and (\ref{eq.sumBf})
that $\sum_{n=1}^{\infty}p_n$ converges.  Therefore, the Borel-Cantelli Lemma tells us that
\[   
  \fl \mathbf{P}\left(  \frac{1}{n} \log Z_n^{Br}(\beta,u) \,\geq \,\frac{1}{n}\log(n+1)\,+
   W+3\epsilon   \quad \hbox{for infinitely many values of $n$}\right) \,=\, 0 \,.
 \]
 Therefore $\limsup_{n\rightarrow\infty} \frac{1}{n} \log Z_n^{Br}(\beta,u)
 \leq W+3\epsilon$ almost surely.  Since $\epsilon$ can be made arbitrarily close to 0, we obtain
 \[
     \limsup_{n\rightarrow\infty} \frac{1}{n} \log Z_n^{Br}(\beta,u) \; \leq\; W    \hspace{5mm} a.s. 
 \]
 
Note also that 
\[e^{\beta u n +S_n}+e^{\beta V((1,2))} Z_n^{[(1,2)]}(\beta) \;\leq\; Z^{Br}_n(\beta,u) \,.
\]
By the Strong Law of Large Numbers (applied to $S_n$), and Equation (\ref{phi}) (recalling
that $Z_n^{[(1,2)]}(\beta)$ and $Z_{n-1}^{HD}(\beta)$ have the same distribution),
we conclude that 
\[  W=\max\{\beta(u+\mu),\phi(\beta)\} \;\leq\;  \liminf_{n\to \infty }\frac{1}{n}\log Z^{Br}_n(\beta,u)
  \hspace{5mm} a.s.
\]
This completes the proof of Theorem~\ref{mainthm}.
\qed

\begin{exmp}
If the disorder distribution is normal with mean $\mu$ and variance $\sigma^2$, then $\beta_c=\sqrt{2 \log d}/\sigma$ and
\begin{equation}
u^{Br}_c(\beta)=\left\{\begin{array}{ll}
\frac{1}{2}\sigma^2\beta+\frac{\log d}{\beta}  &\hbox{ if }\quad  \beta < \beta_c\\
\sigma \sqrt{2\log d}  &\hbox{ if } \quad \beta \geq \beta_c.
\end{array}\right.
\end{equation}
\end{exmp}
\begin{exmp} Let $V$ be a general disorder distribution with mean zero and variance $\sigma^2$. Then $\lambda(\beta)\sim\frac{1}{2}\sigma^2\beta^2$ as $\beta\to 0.$ Therefore, $u^{Br}_c(\beta)\sim\frac{1}{2}\sigma^2\beta+\frac{\log d}{\beta}$ as $\beta\to 0.$
 \end{exmp}

\begin{rem}
More generally, our proofs can be easily modified to show that for the case $d_1=1$ 
and $V(x)\,\stackrel{d}{=}\,\tilde{V}$ for all $x\in\tilde{\T}$, we have 
\[     \lim_{n\rightarrow\infty} \frac{1}{n}\log Z_n^{[0]}(\beta)   \;=\;
       \max\{\beta \mathbf{E}(\tilde{V}),\phi(\beta)\}  \hspace{5mm}a.s. 
\]
This reduces to Theorem \ref{mainthm} in Case I, where $\mathbf{E}(\tilde{V})\,=\,u+\mathbf{E}(V)$.
\end{rem}
\bigskip

\subsection{The defect subtree:  Proof of Theorem~\ref{mainthST} and some auxiliary results.}
\label{subtreesecproof}

For the model with a non-disordered defect subtree, it is not obvious how to prove that the 
limiting free energy exists almost surely.

Therefore, we make the following definitions:
\begin{eqnarray*}
    \overline{f}^{ST}(\beta,u)  & := &   \limsup_{n\rightarrow \infty}  \frac{1}{n}\,\log Z^{ST}_n(\beta,u) \\
    \underline{f}^{ST}(\beta,u)  & := &   \liminf_{n\rightarrow \infty}  \frac{1}{n}\,\log Z^{ST}_n(\beta,u)\,.
\end{eqnarray*}   
By the Kolmogorov zero-one law, 
$\overline{f}^{ST}(\beta,u)$ and $\underline{f}^{ST}(\beta,u)$ are constant almost surely,
so we shall treat $\overline{f}^{ST}$ and $\underline{f}^{ST}$ as deterministic functions.
If $\overline{f}^{ST}(\beta,u)\,=\,\underline{f}^{ST}(\beta,u)$, then we define $f^{ST}(\beta,u)$
to be the common value; in other words, 
\begin{equation}
   \label{eq.fSTdef}
      f^{ST}(\beta,u) \;:=\;   \lim_{n\rightarrow \infty}  \frac{1}{n}\,\log Z^{ST}_n(\beta,u)
     \hspace{5mm}\hbox{if this limit exists.}
 \end{equation}
 
We can now formalize the definition of the three phases that we introduced in Section 
\ref{bulkvsmd}.
\begin{defn}
\label{phase-defn}
We define the three phases as follows:
\begin{eqnarray*}
   {\cal R}_{FP} & := & \{(\beta,u): f^{ST}(\beta,u)=\beta u+\log d_1\}   \hspace{5mm}
         \hbox{(\textit{fully pinned})}    \\
    {\cal R}_{D}  & := & \{(\beta,u): f^{ST}(\beta,u)=\phi(\beta)\}   \hspace{5mm}
         \hbox{(\textit{depinned})}    \\
     {\cal R}_{PP}  & := & \{(\beta,u): \overline{f}^{ST}(\beta,u)>\max\{\phi(\beta),\beta u+\log d_1\} \,\}
        \hspace{5mm}  \hbox{(\textit{partially pinned}).}    
 \end{eqnarray*}
 \end{defn}
 
 \noindent
 Implicit in the definitions is that the limiting free energy of Equation (\ref{eq.fSTdef}) must exist 
 for every point of ${\cal R}_{FP}$ and $\cal{R}_{D}$.
 
Let's recall the definitions of the functions $F$ and $J$ defined in section~\ref{subtreesec}:  
 \begin{eqnarray}
      \label{functionF}
      F(\beta)   & = & \frac{1}{\beta}\,(\lambda(\beta)+\log d -\log d_1)   \\
      \label{functionJ}
      J(\beta) & = &   \frac{1}{\beta}\left( \phi(\beta)-\log d_1-
      \frac{[\lambda(\beta)+\log d-\phi(\beta)]\log d_1}{\lambda(2\beta)-2\lambda(\beta)-\log d}
         \right) .
 \end{eqnarray}

Our characterization of the phases in Theorem~\ref{mainthST} is not complete for $\beta>\beta_c$ when
$u$ is between $F(\beta_c)$ and $J(\beta)$.
However, the following proposition, combined with Theorem~\ref{mainthST},
shows that the partially pinned phase is nonempty, and indeed it
contains points $(\beta,u)$ with $\beta$ arbitrarily close to $\beta_c$.

\begin{prop} 
\label{prop-FJ}
For every $\beta>\beta_c$,
\[F(\beta_c) \,=\,J(\beta_c)\,<\, J(\beta) \,<\, \frac{\phi(\beta)-\log d_1}{\beta} \,<\,F(\beta).\]
\end{prop}

The proof of  Proposition~\ref{prop-FJ} appears at the end of this section,
immediately before the proof of Theorem \ref{mainthST}. We shall first prove some preliminary results.
 
\begin{prop}
   \label{prop-LB0}
For every $\beta>0$ and every $u$,
\[    \liminf_{n\rightarrow \infty}  \frac{1}{n}\,\log Z^{ST}_n(\beta,u)
   \;\geq \;  \max\{\phi(\beta),\   \beta u+\log d_1\}    \quad a.s.
 \]
\end{prop}

Before we prove Proposition~\ref{prop-LB0}, observe that in any path $W$ from the root $\roott$ to generation $n$, there is a node $x$ in $W$ such that 
the part of $W$ from $\roott$ to $x$ is contained in $\tilde{\T}$, and the rest of $W$ is 
outside $\tilde{\T}$.
Writing $k$ to represent the generation of $x$, we see that
\begin{equation}
    \label{eq.Zn0}
  Z^{ST}_n(\beta,u)  \;=\;   \sum_{k=0}^{n-1} \sum_{x\in \tilde{T}: \,\textrm{Height}(x)=k}
       \sum_{y\in D(x)} e^{\beta k u} e^{\beta V(y)} Z_n^{[y]}(\beta)   \;+\;  d_1^ne^{\beta nu}.
\end{equation}
(The rightmost term corresponds to those paths that never leave $\tilde{\T}$, i.e.\ $k=n$.)
\begin{proof}[Proof of Proposition~\ref{prop-LB0}]
We shall use Equation (\ref{eq.Zn0}).  Specifically, for $y\in D(\roott)$, we have
\[    Z^{ST}_n(\beta,u)  \;\geq \;  e^{\beta V(y)} Z_n^{[y]}(\beta) 
     \hspace{5mm}\hbox{and}\hspace{5mm}    
  Z_n^{[y]}(\beta) \;\eqd\;  Z_{n-1}^{HD}(\beta) \,.
\]
Then by Equation~(\ref{phi}), we have
\begin{equation}  
  \label{eq.LB01}
      \liminf_{n\rightarrow \infty}  \frac{1}{n}\,\log Z^{ST}_n(\beta,u)
   \;\geq \;  \phi(\beta)     \hspace{5mm}  \hbox{a.s.}
 \end{equation}
Also, since $Z^{ST}_n(\beta,u)  \,\geq \;d_1^ne^{\beta nu}$, we have 
\begin{equation}  
  \label{eq.LB02}
    \liminf_{n\rightarrow \infty}  \frac{1}{n}\,\log Z^{ST}_n(\beta,u)  
   \;\geq \;  \log d_1  \,+\,  \beta u \hspace{5mm}  \hbox{a.s.}
\end{equation}
The proposition follows from inequalities (\ref{eq.LB01}) and (\ref{eq.LB02}).
\end{proof}
\bigskip
We introduce the following notation:  for $0\leq k< n$, 
\begin{equation}
   \label{eq.Gdef}
    G^{ST}_{k,n}(\beta)  \;:=\;  \sum_{x\in \tilde{T}: \,\textrm{Height}(x)=k }
       \sum_{y\in D(x)} e^{\beta V(y)} Z_n^{[y]}(\beta)  \,.
\end{equation}

(Observe that in the case $d_1=1$, the above would reduce to $G_{k,n}^{Br}(\beta)$ as
defined in Equation (\ref{eq.GBR1}).)
Then we see from Equation (\ref{eq.Zn0}) that
\begin{equation}
    \label{eq.ZGdecomp}
 Z^{ST}_n(\beta,u)   \;=\;   \sum_{k=0}^{n-1} e^{\beta ku} G^{ST}_{k,n}(\beta)   \;+\;  d_1^ne^{\beta nu},
\end{equation}
which is the analogue of Equation (\ref{eq.ZsumG1}).
Since $G^{ST}_{k,n}(\beta)$ is a sum of $d_1^{k}(d-d_1)$ independent copies of 
$e^{\beta V^*}Z_{n-k-1}^{HD}(\beta)$ (where $V^*$ is a copy of $V$, independent of everything else), 
we have
\begin{equation}
   \label{eq.Gmean}
      \mathbf{E}(G^{ST}_{k,n}(\beta))   \;=\;  d_1^{k}(d-d_1)d^{n-k-1} e^{(n-k)\lambda(\beta)}
      \hspace{5mm}\hbox{and}
 \end{equation}
 \begin{equation}
    \label{eq.Gvar}
      \textrm{Var}(G^{ST}_{k,n}(\beta))   \;=\;  d_1^{k}(d-d_1)
       \textrm{Var}(e^{\beta V^*} Z^{HD}_{n-k-1}(\beta)) \,.
 \end{equation}

\begin{prop}
    \label{prop-UB0}
For every $\beta>0$ and every $u$, 
\[    \limsup_{n\rightarrow \infty}  \frac{1}{n}\,\log Z^{ST}_n(\beta,u)  
   \;\leq \;  \max\{\lambda(\beta)+\log d,\   \beta u+\log d_1\}  \hspace{5mm}  \hbox{a.s.} 
\]
\end{prop}
\begin{proof}[Proof of Proposition~\ref{prop-UB0}]
For given $(\beta,u)$, let $M \,=\,\max\{\lambda(\beta)+\log d,\   \beta u+\log d_1\}$.   Fix $C>M$.
By Equation (\ref{eq.ZGdecomp}), we have
\begin{eqnarray}
   & \mathbf{P} \left(  \frac{1}{n} \log Z^{ST}_n(\beta,u) \,>\,C\right)      \;=\;   
         \mathbf{P}\left( Z^{ST}_n(\beta,u)\,>\,e^{Cn}\right)    \\
   & \hspace{10mm} \leq \; 
     \sum_{k=0}^{n-1} \mathbf{P} \left(e^{\beta ku} G^{ST}_{k,n}(\beta)\,>\,  \frac{e^{Cn}}{n+1}\right)
         \,+\,    {\bf 1}\left(d_1^n e^{\beta nu}\,>\,\frac{e^{Cn}}{n+1}\right) \,.
      \label{eq.ZsumP}
\end{eqnarray}
The rightmost term in Equation (\ref{eq.ZsumP}) is zero for all sufficiently large $n$, since
$C> \beta u +\log d_1$.  For $0\leq k\leq n-1$, we have
\begin{eqnarray*}
  \mathbf{P} &\left(e^{\beta ku} G^{ST}_{k,n}(\beta)\,>\,  \frac{e^{Cn}}{n+1}\right)   
    \\
    &  \leq \;  \frac{e^{\beta ku}\,\mathbf{E}(G^{ST}_{k,n}(\beta))}{e^{Cn}/(n+1)}
          \hspace{18mm}\hbox{(by Markov's Inequality)}     \\
     & < \; \frac{e^{\beta ku}\,d_1^{k}  \,d^{n-k}\,e^{(n-k)\lambda(\beta)}}{e^{Cn}/(n+1)}  
        \hspace{15mm}\hbox{(by Equation \ (\ref{eq.Gmean}))}  \\     
    & = \; (n+1) 
           \,\frac{(e^{\lambda(\beta)+\log d})^{n-k}(e^{\beta u+\log d_1})^k}{e^{Cn} }   \\
     & \leq \; (n+1)
       \,e^{n(M-C)}  \,.
 \end{eqnarray*}     
Hence, for large $n$, 
\[    \mathbf{P} \left(  \frac{1}{n} \log Z^{ST}_n(\beta,u)  \,>\,C\right)   \;\leq \;   
     n(n+1)  
   \,e^{n(M-C)}  \,.
\]
Since $M-C<0$, the Borel-Cantelli Lemma shows that, with probability 1, there are only finitely
many values of $n$ for which $\frac{1}{n} \log Z^{ST}_n(\beta,u)  \,>\,C$.  This proves that
\[\limsup_{n\rightarrow \infty}  \frac{1}{n}\,\log Z^{ST}_n(\beta,u) \,\leq \,C \hbox{ a.s. }\]
Since $C$ can be 
arbitrarily close to $M$, Proposition~\ref{prop-UB0} follows.

\end{proof}
\bigskip

Since $\phi(\beta)=\lambda(\beta)+\log d$ for $\beta\leq \beta_c$, 
Propositions \ref{prop-LB0} and \ref{prop-UB0} immediately imply the following.

\medskip
\begin{cor}
   \label{subtreecor1}
For every $\beta\leq \beta_c$, 
\[    \lim_{n\rightarrow \infty}  \frac{1}{n}\,\log Z^{ST}_n(\beta,u)  
   \; = \;  \max\{\phi(\beta),\   \beta u+\log d_1\} \,.
 \]
 In particular, the limit exists almost surely.
\end{cor}

Since $\phi(\beta)\,<\,\lambda(\beta)+\log d$ for $\beta > \beta_c$ (Equation (\ref{eq.philtlam})), 
we must ask whether
the conclusion of Corollary \ref{subtreecor1} holds for all values of $\beta$.  We shall show 
that the answer is no.
This is interesting because it tells us that the dominant terms in the partition function 
are neither the walks that spend almost all their time in the defect subtree nor the walks
that spend hardly any time in the defect subtree.  This is in direct contrast to the case of 
a defect branch ($d_1=1$), which we examined in Section \ref{dfbranchsec}.

\medskip

The next lemma will be needed for an application of Chebychev's Inequality.
\begin{lemma}
   \label{lem.Theta}
(a) For every $\beta \geq 0$, the limit 
\[     \Theta \;\equiv \Theta(\beta) \;:=\; \lim_{n\rightarrow\infty}\left(  
    \hbox{Var}(e^{\beta V^*} Z^{HD}_n(\beta))\right)^{1/n}   
\]
exists and equals $ \max\{d^2e^{2\lambda(\beta)},d e^{\lambda(2\beta)}\}$.

\smallskip
\noindent
(b)  For every $\beta>\beta_c$, we have $\Theta \,=\, d e^{\lambda(2\beta)}$.
\end{lemma}
\begin{proof}[Proof of Lemma~\ref{lem.Theta}]
$(a)$ Let us write $Y=e^{\beta V^*}$ and $Z_n =Z_n^{HD}(\beta)$. Recall that $V^*$ is a copy of $V$, independent of everything else. Since $Y$ and 
$Z_n$ are independent, it is easy to see that 
$\textrm{Var}(YZ_n)\,=\, \mathbf{E}(Z_n^2)\textrm{Var}(Y)+(\mathbf{E}(Y))^2\textrm{Var}(Z_n)$, and hence that
\[   \mathbf{E}(Z_n^2)\textrm{Var}(Y)  \;\leq \;  \textrm{Var}(YZ_n)  \;\leq \;  
    \mathbf{E}((YZ_n)^2)  \;=\; \mathbf{E}(Y^2)\mathbf{E}(Z_n^2) \,. 
\]
Since $Y$ does not depend on $n$, we deduce that  
$\Theta =\lim_{n\rightarrow\infty}\mathbf{E}(Z_n^2)^{1/n}$ if this limit exists.  
Using this observation,  
part (a) follows from the following calculation:
\begin{eqnarray*}
  \mathbf{E}[(Z^{HD}_n(\beta))^2]&=& \sum_{{W:\roott\rightarrow (n,\cdot)}}
    \sum_{{W':\roott\rightarrow (n,\cdot)}}
    \mathbf{E}[e^{\beta V\langle W\rangle+\beta V\langle W'\rangle}]\\
       &=&d^n\left(\sum_{k=0}^{n-1}(d-1)d^{n-k-1}(e^{\lambda(2\beta)})^k (e^{2\lambda(\beta)})^{n-k}
       +e^{n\lambda(2\beta)}\right)\\
   & = & d\,(d-1)\,e^{2\lambda(\beta)}
    \sum_{k=0}^{n-1}  \left(de^{\lambda(2\beta)}\right)^k
   \left(d^2e^{2\lambda(\beta)}\right)^{n-1-k} +\left(de^{\lambda(2\beta)}\right)^n   \\
     & = &    d(d-1)  
     \,e^{2\lambda(\beta)}\, 
      \frac{(de^{\lambda(2\beta)})^n-(d^2e^{2\lambda(\beta)})^n}{
    de^{\lambda(2\beta)}-d^2e^{2\lambda(\beta)} }
       +\left(de^{\lambda(2\beta)}\right)^n.
\end{eqnarray*}

$(b)$ Fix $\beta>\beta_c$, and assume that $d^2e^{2\lambda(\beta)}>de^{\lambda(2\beta)}$.  
That is, $\log d> \lambda(2\beta)-2\lambda(\beta)$. 
By the Mean Value Theorem, we know that $\lambda(2\beta)=\lambda(\beta)+\beta \lambda{'}(\tilde \beta)$ 
for some $\beta<\tilde \beta<2\beta.$
Recalling Equation~(\ref{thatf}) and Lemma~\ref{criticalbeta}, we find that
\begin{eqnarray*}
f(\beta)&>& \lambda(\beta)+\lambda(2\beta)-2\lambda(\beta)-\beta \lambda'(\beta)\\
&=&\beta (\lambda'(\tilde \beta)-\lambda'(\beta))\\
&>&0    \hspace{22mm}\hbox{(since $\lambda''>0$).}
\end{eqnarray*} 
This contradicts the fact that $f(\beta)<0$ for $\beta\in (\beta_c,\infty)$.
Therefore $d^2e^{2\lambda(\beta)}\leq de^{\lambda(2\beta)}$.  Part (b) follows.
\end{proof}

The following lemma plays an important role in proving that the partially pinned phase is not empty for $\beta>\beta_c$. 
\begin{lemma}
\label{lem-LBt1}
Fix $\beta>0$.  Let $t$ be a real number in $(0,1)$ such that 
\begin{equation}
    \label{eq.tdef}
        \left(\frac{\Theta}{d^2e^{2\lambda(\beta)}} \right)^{1-t} \;<\,  d_1^t \,.
 \end{equation}
Then for every $u$,
\begin{equation}
  \label{eq.LBt1conc}
      \liminf_{n\rightarrow \infty}  \frac{1}{n}\,\log Z^{ST}_n(\beta,u)  
   \;\geq \;  (\lambda(\beta)+\log d)(1-t) \,+\,   (\beta u+\log d_1)t  \hspace{3.5mm}  \hbox{a.s. } 
\end{equation}
\end{lemma}

\begin{rem}
   \label{rem.LBt1}
   Observe that Equation (\ref{eq.tdef}) holds when $t$ is close enough to 1.
\end{rem}

Before we prove Lemma \ref{lem-LBt1}, we shall show how it can be used.

\begin{prop}
   \label{prop-ineqex}
Assume $\beta>\beta_c$.  Then the strict inequality 
\begin{equation}
   \label{eq.strictineq}
  \liminf_{n\rightarrow \infty}  \frac{1}{n}\,\log Z^{ST}_n(\beta,u) 
   \; >\;  \max\{\phi(\beta),\   \beta u+\log d_1\}    \hspace{7mm}\hbox{a.s.}
\end{equation}   
holds  if either of the following  hold:
\\
(a)  $\phi(\beta) \,=\,  \beta u +\log d_1$; or
\\
(b)  $\lambda(\beta)+\log d \,>\,  \beta u +\log d_1 \,> \,  \phi(\beta)$.
\\
In particular,  $(\beta,u)$ is in the partially pinned phase ${\cal R}_{PP}$ if
$(\phi(\beta)-\log d_1)/\beta  \,\leq \,u\,<\,F(\beta)$.
\end{prop}

The proof of Theorem \ref{mainthST} will also use Lemma \ref{lem-LBt1} to prove that 
the inequality (\ref{eq.strictineq}) also holds if $\beta u+\log d_1$ is smaller than, but 
sufficiently close to, $\phi(\beta)$.

\begin{proof}[Proof of Proposition \ref{prop-ineqex}:] 
Let $M\,=\, \max\{\phi(\beta),\   \beta u+\log d_1\}$.

\smallskip
\noindent
(a)  In this case, $M\,=\, \phi(\beta) \,=\,  \beta u +\log d_1$.  Since $\phi(\beta)<\lambda(\beta)+\log d$
(by Equation (\ref{eq.philtlam})), 
we see that the right side of Equation (\ref{eq.LBt1conc}) is strictly greater than $M$ for every $t$ 
in the interval $(0,1)$.  By Remark \ref{rem.LBt1}, the result follows.

\smallskip
\noindent
(b)  In this case, $\lambda(\beta)+\log d \,>\, M \,=\, \beta u+\log d_1$.  As in part (a), the result 
follows.
\end{proof}

\begin{proof}[Proof of Lemma \ref{lem-LBt1}:]
For $0\leq k< n$, define the event 
\[    {\cal A}[k,n]  \;:=   \;  
   \left\{  \left|G^{ST}_{k,n}(\beta)-\mathbf{E}(G^{ST}_{k,n}(\beta)) \right| \,\geq \,
     \frac{1}{2}\, \mathbf{E}(G^{ST}_{k,n}(\beta))  \,\right\}  \,.
\]
By Chebychev's Inequality, and Equations~(\ref{eq.Gmean}) and (\ref{eq.Gvar}), we have 
\begin{eqnarray*}
     \mathbf{P}({\cal A}[k,n]) & \leq & 
        \frac{4\, \textrm{Var}(G^{ST}_{k,n}(\beta))}{(\mathbf{E}(G^{ST}_{k,n}(\beta))^2}    \\
       & =  &   \frac{4\, \textrm{Var}(e^{\beta V^*} Z_{n-k-1}^{HD}(\beta))}{d_1^k(d-d_1) 
            d^{2(n-k-1)} e^{2(n-k)\lambda(\beta)} }
 \end{eqnarray*}
 
 Next, we let $k=k(n)$ be an integer-valued function of $n$ with the property that
 \[  \lim_{n\rightarrow\infty} \frac{k(n)}{n}   \; = \; t \,     \]    
 where $t$ is given in the statement of the Lemma.  Then
 \begin{equation}
     \label{eq.PrAlim}
     \limsup_{n\rightarrow\infty} \mathbf{P}({\cal A}[k(n),n])^{1/n}   \;\leq \; 
       \frac{\Theta^{1-t}}{d_1^t (de^{\lambda(\beta)})^{2(1-t)} }    \;<\; 1  \,
 \end{equation}
 where the final inequality is a consequence of Equation (\ref{eq.tdef}).   Therefore
$\mathbf{P}({\cal A}[k(n),n])$ decays to 0 exponentially rapidly in $n$, and 
 hence the Borel-Cantelli Lemma shows that (with probability 1) ${\cal A}[k(n),n]$  occurs for
 only finitely many values of $n$. 
      
  Observe that $G^{ST}_{k,n}(\beta) > \mathbf{E}(G^{ST}_{k,n}(\beta))/2$   on ${\cal A}[k,n]^c$       
  (where ${}^c$ denotes complement).  Therefore
\[       G^{ST}_{k(n),n}(\beta)    \;\geq \;    \frac{1}{2} \,\mathbf{E}(G^{ST}_{k,n}(\beta)) 
   \,{\bf 1}({\cal A}[k(n),n]^c) \,.
\]
The final conclusion of the previous paragraph, together with Equation (\ref{eq.Gmean}), shows
that 
\begin{equation}
   \label{eq.Gknt}
        \liminf_{n\rightarrow \infty} G^{ST}_{k(n),n}(\beta)^{1/n}   \;\geq \;  
          d_1^t(de^{\lambda(\beta)})^{1-t}   \hspace{5mm}\hbox{a.s.}
\end{equation}   
Finally, Equation (\ref{eq.ZGdecomp})  implies that $Z^{ST}_n(\beta,u) \,\geq \,
e^{\beta k(n)u}G^{ST}_{k(n),n}(\beta)$.
Lemma~\ref{lem-LBt1} follows immediately from this and  Equation (\ref{eq.Gknt}).
\end{proof}

We define the critical curve as
\begin{equation}
u^{ST}_c(\beta):=\inf\{u\in \mathbb{R}: \overline{f}^{ST}(\beta,u)>\phi(\beta)\}.
\end{equation}

Then we have the following.
\bigskip
\begin{prop}
  \label{prop-desorblargebeta}
  Assume $\beta \geq \beta_c$ and $u\leq \Psi$, where
\[    \Psi   \;:=\;   \frac{1}{\beta_c}(\lambda(\beta_c)+\log d-\log d_1) \,.
\]
Then
\begin{equation}
     \label{eq.limisphi}
     \lim_{n\rightarrow\infty} \frac{1}{n}\log Z^{ST}_n(\beta,u)    \;=\; \phi(\beta)   \hspace{5mm}a.s.
 \end{equation}
That is, $u^{ST}_c(\beta)\,\geq \,\Psi$ for all $\beta\geq \beta_c$.
\end{prop}

\medskip
\noindent
\begin{proof}[Proof of Proposition~\ref{prop-desorblargebeta}]
Fix $\beta \geq \beta_c$ and $u\leq \Psi$.   Let $\epsilon>0$ and let
\[     C\;:=\;  \phi(\beta)\,+\,\epsilon  \;=\; \frac{\beta}{\beta_c}(\lambda(\beta_c)+\log d)  \,+\,\epsilon \,.
\]
Our approach is similar to the proof of Proposition \ref{prop-UB0}.  As in Equation (\ref{eq.ZsumP}),
we have  
\begin{eqnarray}
   & \mathbf{P} \left(  \frac{1}{n} \log Z_n^{ST}(\beta,u)  \,>\,C\right) \;=\;
         \mathbf{P} \left(  \sum_{k=0}^n e^{\beta ku}G^{ST}_{k,n}(\beta) \,>\,e^{Cn}\right)
          \nonumber  \\
   & \hspace{10mm} \leq \;
     \sum_{k=0}^{n-1} \mathbf{P}  \left(e^{\beta ku} G^{ST}_{k,n}(\beta)\,>\,
        \frac{e^{Cn}}{n+1}\right)
         \,+\,    {\bf 1}\left(d_1^n e^{\beta nu}\,>\,\frac{e^{Cn}}{n+1}\right) \,.
      \label{eq.ZsumP1}
\end{eqnarray}
Since $\beta u+\log d_1\,\leq \,\beta \Psi+(\beta/\beta_c)\log d_1\,<\,C$, the final term in Equation (\ref{eq.ZsumP1}) is 0 for all
sufficiently large $n$.

Since $\beta\geq \beta_c$, we have $G^{ST}_{k,n}(\beta)^{1/\beta}\,\leq \, 
G^{ST}_{k,n}(\beta_c)^{1/\beta_c}$ by Lemma 5 of \cite{BPP},  and thus
\begin{eqnarray*}
  \mathbf{P}  &\left(e^{\beta ku} G^{ST}_{k,n}(\beta)\,>\,  \frac{e^{Cn}}{n+1}\right)
    \\
    & \leq \;   \mathbf{P}  \left(G^{ST}_{k,n}(\beta)^{\beta/\beta_c}\,>\, 
    \frac{e^{Cn-\beta ku}}{n+1}\right)
   \\
    & = \;   \mathbf{P}  \left(G^{ST}_{k,n}(\beta)\,>\,
       \frac{e^{(\beta_c/\beta)Cn-\beta_c ku}}{(n+1)^{\beta_c/\beta}}\right)
        \\
    &  \leq \;  \frac{(n+1)^{\beta_c/\beta}e^{\beta_c ku}\,
    \mathbf{E} (G^{ST}_{k,n}(\beta))}{e^{(\beta_c/\beta)Cn}}
          \hspace{18mm}\hbox{(by Markov's Inequality)}     \\
     & = \; \frac{(n+1)^{\beta_c/\beta}e^{\beta_c ku}\,d_1^{k} (d-d_1)d^{n-k-1}
     e^{(n-k)\lambda(\beta_c)}}{e^{(\beta_c/\beta)Cn}}
        \hspace{8mm}\hbox{(by Equation \ (\ref{eq.Gmean}))}
        \\
    & \leq \; (n+1)   \, \left(\frac{e^{\lambda(\beta_c)+\log d} }{e^{C\beta_c/\beta}}\right)^{n-k}
            \left(\frac{e^{\beta_c u+\log d_1}}{e^{C\beta_c/\beta}}\right)^k  \,.
 \end{eqnarray*}
Since
\[    \frac{C\beta_c}{\beta}  \;=\;   \lambda(\beta_c)+\log d +\frac{\beta_c\epsilon}{\beta}
      \;= \;   \beta_c\Psi  +\log d_1 + \frac{\beta_c\epsilon}{\beta} 
      \;\geq \;   \beta_cu+\log d_1 +\frac{\beta_c\epsilon}{\beta}\,,
\]
we conclude that for large $n$,
\[    \mathbf{P} \left(  \frac{1}{n} \log Z_n^{ST}(\beta)  \,>\,C\right)   \;\leq \;
 n(n+1)\,e^{-n\epsilon \beta_c/\beta}  \,.
\]
It follows from the Borel-Cantelli Lemma that
$\limsup_{n\rightarrow\infty}\frac{1}{n}\log Z_n^{ST}(\beta,u) \,\leq \,\phi(\beta)+\epsilon$ with
probability 1.  Since this holds for every positive $\epsilon$, we can combine this with
Proposition \ref{prop-LB0} to complete the proof of Equation (\ref{eq.limisphi}).
\end{proof}

We are now ready to prove the main results of this section.

\begin{proof}[Proof of Proposition~\ref{prop-FJ}]  
Note that $\lambda(\beta_c)+\log d=\phi(\beta_c)$ and hence $F(\beta_c)=J(\beta_c)$. 
Fix $\beta> \beta_c$.  We have $\lambda(\beta)+\log d>\phi(\beta)$ (by Equation (\ref{eq.philtlam})) 
and 
\begin{equation}
   \label{eq.lamlam}
   \lambda(2\beta)-2\lambda(\beta)-\log d>0
\end{equation}
by Lemma \ref{lem.Theta}(a,b).
Hence the second and third inequalities follow.

It remains to prove the first inequality in the proposition.  
Since $\lambda$ is strictly convex, we can consider slopes of secant lines to obtain
\begin{equation}
    \label{eq.slopes}
     \frac{\lambda(2\beta)-\lambda(\beta)}{\beta}   \;>\;
       \frac{\lambda(\beta)-\lambda(\beta_c)}{\beta-\beta_c}  \,.
\end{equation}   
Now replace $\phi(\beta)$ by $\frac{\beta}{\beta_c}(\lambda(\beta_c)+\log d)$ in the definition
of $J(\beta)$ in Equation~(\ref{functionJ}), obtaining
\begin{eqnarray*}
   J(\beta) & = & \frac{\lambda(\beta_c)+\log d}{\beta_c}  \,-\, \frac{\log d_1}{\beta}\left( 1\,+\,
       \frac{\lambda(\beta)+\log d-\phi(\beta)}{\lambda(2\beta)-2\lambda(\beta)-\log d}\right)   \\
       & = & F(\beta_c)+\frac{\log d_1}{\beta_c} -
       \frac{\log d_1}{\beta}\left( \frac{\lambda(2\beta)-\lambda(\beta)-\phi(\beta)
             }{\lambda(2\beta)-2\lambda(\beta)-\log d}\right) \,.       
\end{eqnarray*}
Now, some algebra gives
\begin{eqnarray*}
\frac{J(\beta)-F(\beta_c)}{\log d_1}  & = & \frac{1}{\beta_c}  -\frac{1}{\beta}
     \left( \frac{\lambda(2\beta)-\lambda(\beta)-(\beta/\beta_c)[\lambda(\beta_c)+\log d]
             }{\lambda(2\beta)-2\lambda(\beta)-\log d}\right)     \\
   & = & \frac{(\beta-\beta_c)[\lambda(2\beta)-\lambda(\beta)] -\beta[\lambda(\beta)-\lambda(\beta_c)]
      }{\beta \beta_c  [\lambda(2\beta)-2\lambda(\beta)-\log d]}
      \\
      & > & 0   \hspace{18mm}\hbox{(by Equations (\ref{eq.lamlam}) and (\ref{eq.slopes})).}
\end{eqnarray*}
Therefore $J(\beta)>F(\beta_c)$, and the proof is complete.
\end{proof}

\begin{proof}[Proof of Theorem \ref{mainthST}]
We start with the observation that every point $(\beta,u)$ (with $\beta>0$) is in at least one
of ${\cal R}_{FP}$ or ${\cal R}_{D}$ or ${\cal R}_{PP}$.  
To see this, suppose  $(\beta,u)\not\in{\cal R}_{PP}$. Then 
$\max\{\phi(\beta),\beta u+\log d_1\}\geq \overline{f}^{ST}(\beta,u)\geq \underline{f}^{ST}(\beta,u)$.
But $\underline{f}^{ST}(\beta,u)\geq \max\{\phi(\beta),\beta u+\log d_1\}$ by Proposition \ref{prop-LB0},
so the limit $f^{ST}(\beta,u)$ exists  and equals  $\beta u+\log d_1$ or $\phi(\beta)$.  That is,
$(\beta,u)\in {\cal R}_{FP}\cup{\cal R}_{D}$.

To prove part (a),  fix  $\beta\leq \beta_c$.
By    Corollary \ref{subtreecor1}, the limiting free energy exists and is given by
\[    f^{ST}(\beta,u)    \;=\;   \max\{\phi(\beta),\beta u+\log d_1\}  \,.   \]
First assume $u\geq F(\beta)$. This is equivalent to  
$\beta u+\log d_1 \geq \lambda(\beta)+\log d$.  Since 
$\phi(\beta)=\lambda(\beta)+\log d$, we obtain $f^{ST}(\beta,u)=\beta u+\log d_1$, and 
hence $(\beta,u)\in {\cal R}_{FP}$.  Similarly, the assumption $u\leq F(\beta)$ leads
to $\beta u+\log d_1\leq\lambda(\beta)+\log d=\phi(\beta)$, and hence $f^{ST}(\beta,u)=\phi(\beta)$,
i.e.\ $(\beta,u)\in {\cal R}_{D}$.

Now we shall prove part (b).  Fix $\beta >\beta_c$.

First assume $u\geq F(\beta)$.  Then $\beta u+\log d_1\geq \lambda(\beta)+\log d \geq  \phi(\beta)$
(the second inequality is from Equation (\ref{eq.phileqlam})).  By Propositions \ref{prop-LB0}
and \ref{prop-UB0}, we have
\[    \underline{f}^{ST}(\beta,u) \;\geq \;  \beta u+\log d_1 \;\geq \; \overline{f}^{ST}(\beta,u)  \,, \]
which shows that $f^{ST}$ exists and equals $\beta u+\log d_1$.  Therefore $(\beta,u)\in {\cal R}_{FP}$.

Next, assume that $u\leq F(\beta_c)$.  Then Proposition \ref{prop-desorblargebeta} says that
$(\beta,u)\in {\cal R}_{D}$, since $F(\beta_c)=\Psi$.

Next, assume that $J(\beta)<u<F(\beta)$.   Let 
\begin{eqnarray}
    \label{eq.deftstar}
    t^* \;\equiv \; t^*(\beta)  & = &   \frac{\lambda(2\beta)-2\lambda(\beta)-\log d}{
         \log d_1+\lambda(2\beta)-2\lambda(\beta)-\log d}     \hspace{5mm}\hbox{and}   \\
         \label{eq.defLbt}
         L(\beta,t)  & = & \phi(\beta)-\left(\frac{1-t}{t}\right)(\lambda(\beta)+\log d-\phi(\beta))
            \hspace{5mm}\hbox{for $t\neq 0$}.
\end{eqnarray}
By Lemma \ref{lem.Theta}(a,b), we know that $\lambda(2\beta)-2\lambda(\beta)-\log d>0$,
which implies that $t^*\in (0,1)$. Also, observe  that 
\begin{equation}
    \label{eq.JL}
         J(\beta)   \;=\; \frac{L(\beta,t^*)-\log d_1}{\beta} \,.   
\end{equation}
Recalling from Lemma \ref{lem.Theta}(b) that $\Theta =de^{\lambda(2\beta)}$, it is easy to 
show that 
\[     
        \left(\frac{\Theta}{d^2e^{2\lambda(\beta)}} \right)^{1-t^*} \;=\,  d_1^{\,t^*} \,,
 \]
and that the inequality of Equation (\ref{eq.tdef}) holds whenever $t^*<t<1$.

Still assuming  $J(\beta)<u<F(\beta)$, we consider two possible cases:
\begin{verse}
($i$)   $\phi(\beta)\leq \beta u+\log d_1$, or
\\
($ii$)   $\beta u+\log d_1<\phi(\beta)$.
\end{verse}
If case ($i$) holds, then Proposition \ref{prop-ineqex} says that
$(\beta,u)\in{\cal R}_{PP}$.  So we shall assume that case ($ii$) holds.  
Since 
$(L(\beta,t^*)-\log d_1)/\beta<u$, we can choose $t\in (t^*,1)$ such that 
$(L(\beta,t)-\log d_1)/\beta<u$.   From simple algebra, it follows that, for this value of $t$, 
the right hand side
of Equation (\ref{eq.LBt1conc}) is strictly greater than $\phi(\beta)$, which 
in turn equals $\max\{\phi(\beta),\beta u+\log d_1\}$ in case ($ii$).  
For this choice of $t$, Lemma \ref{lem-LBt1} shows
that  $(\beta,u)\in{\cal R}_{PP}$.

Finally, assume that $F(\beta_c)<u\leq J(\beta)$.  
Since $\phi(\beta)<\lambda(\beta)+\log d$  (by Equation (\ref{eq.philtlam})), 
we see from Equation (\ref{eq.defLbt}) that $L(\beta,t^*)<\phi(\beta)$.
We then have $u\leq J(\beta)<(\phi(\beta)-\log d_1)/\beta$ (by Equation (\ref{eq.JL})), and hence
$\beta u+\log d_1<\phi(\beta)$.  Therefore, by Proposition \ref{prop-LB0}, 
we have $\underline{f}^{ST}(\beta,u)\geq \phi(\beta)>\beta u+\log d_1$.  Hence $(\beta,u)$ is not 
in ${\cal R}_{FP}$, so it must be in ${\cal R}_{D}$ or ${\cal R}_{PP}$ by the first observation of the
present proof.
\end{proof}

\section{Discussion of the Results and Some Future Research Directions}
\label{finalsec}
In this section, we will discuss our results and also mention some possible future research directions.

\subsection{Polymers on disordered trees with a shifted-disordered defect subtree}
Let's assume that $\tilde V(x)=V(x)+u$ for $x\in \tilde \T$, so that  the Hamiltonian is given by  
\[V\langle W\rangle   \;:=\;  \sum_{y\in \tilde{\T}\cap (W\setminus \{\roott\})} (V(y)+u)+\sum_{y\in (\T\setminus \tilde \T)\cap (W\setminus \{\roott\})} V(y)\,.    \]
Note that the localized microscopic defect was non-random in the model introduced in section~\ref{subtreesec}. 

We denote the partition function of this model by
\begin{equation}
\label{eq.subtreepf2}
\tilde Z_n^{ST}(\beta,u):=Z^{[\roott]}_n(\beta).
\end{equation}
Then for $u\geq 0$, we have
\begin{equation}
\label{eq.lowerbnd}
\liminf_{n\to \infty}\frac{1}{n}\log \tilde Z_n^{ST}(\beta,u) \geq \max(\phi(\beta),\beta u+\tilde \phi(\beta))  \hspace{5mm}a.s.
\end{equation}
where 
\begin{equation}
\label{eq.qfreetilde}
\tilde \phi(\beta)=\left\{\begin{array}{ll}
\lambda(\beta) +\log d_1 &\hbox{ if } \quad \beta <\beta_c\\
\frac{\beta}{\beta_c}(\lambda(\beta_c)+\log d_1) &\hbox{ if }\quad  \beta \geq \beta_c.
\end{array}\right.
\end{equation}
To see that the left side of the Equation~(\ref{eq.lowerbnd}) is greater than $\phi(\beta)$,  we restrict the partition function to $\T \setminus \tilde \T$; and to see that it is greater than $\beta u+\tilde \phi(\beta)$,
we restrict the partition function to $\tilde \T$.

We don't yet know much more about this general model but it
is reasonable to suspect a  phase diagram similar to Figure~\ref{phasediagramfig}.

\subsection{Directed polymers on disordered integer lattice $\mathbb{Z}^d$ with a defect line} 
\label{latticemodel}
The $1+d$ dimensional \textit{lattice} version of the directed polymer in a random environment is formulated as follows.

\begin{figure}
\centering
\begin{tikzpicture}[scale=0.8]
\draw [help lines, dashed] (0,0) grid (10,6);
\draw [->] (0,3) -- (10.5,3);
\draw [->] (0,0) -- (0,6.5);

\draw [ultra thick] (0,3)--(1,4)--(3,2)--(4,3)--(5,2)--(8,5)--(9,4)--(10,5);

\foreach \a in {0,1,...,10}
{
  \foreach \b in {0,1,...,6}
  {
      \filldraw[opacity=rnd] (\a, \b) circle(2pt);
  }
}

\foreach \c in {0,1,...,10}
{\draw (\c,3) circle(4.5pt);}
\end{tikzpicture}
\caption{A $1+1$ dimensional directed polymer in a random environment with a defect line. The polymer configurations are represented by directed random walk paths. Each site of the lattice $\mathbb{Z}^2$ is assigned a random variable which represent the \textit{bulk disorder}. The sites on the $x$-axis carries an extra potential $u$ which represents the \textit{defect line}. It is not clear for a given $\beta$ whether $u_c(\beta)>0$ or $u_c(\beta)=0$, defined so that for $u>u_c(\beta)$ the polymer places a positive fraction of its monomers on the $x$-axis.}  
\end{figure}
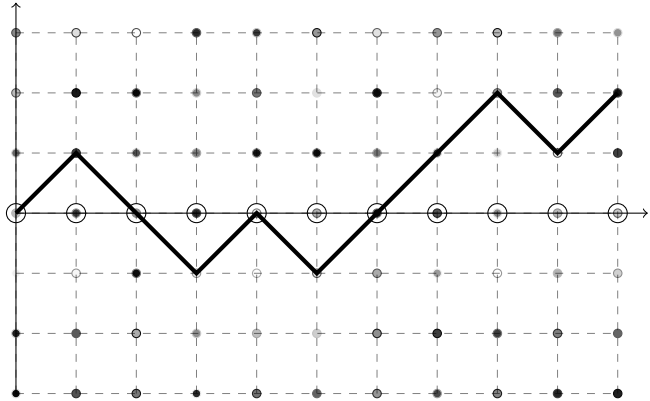

The polymer configurations are represented by the directed paths of the simple symmetric random walk (SSRW) $\{(j, S_j)\}_{j=1}^{n}$ in $\mathbb{N} \times \mathbb{Z}^d$.  The disordered random environment is given by i.i.d.~random variables $\{v(i,x) : i \geq 1, x \in \mathbb{Z}^d\}$ with law denoted by $\mathbf{P}$ satisfying 
\begin{equation}\label{expm1}
\lambda(\beta)=\log \mathbf{E}[e^{\beta v(i,x)}]< \infty \quad \hbox{for all}  \quad \beta \in \mathbb{R}.
\end{equation}
The Hamiltonian of the model is given by
\begin{equation}
  H_n(S):=\sum_{j=1}^{n}v(j, S_j)
\end{equation}
and $Z_n(\beta):=\sum_{S}e^{\beta H_n(S)}$
denotes the \textit{partition function} where the sum is overall SSRW paths of length $n$ with $S_0=0$. The free energy of the model is defined as
\begin{equation}
\label{Lqfree}
f(\beta):=\lim_{n\to \infty}\frac{1}{n}\log Z_n(\beta).
\end{equation}
The existence of the free energy is first proven by Carmona and Hu \cite{CaHu1} for the Gaussian environment and then for any distribution which satisfies the exponential moment condition in Equation~(\ref{expm1}) by Comets et al. in \cite{CSY1}. There is no explicit expression for the free energy for the lattice case as opposed to the tree case as in Equation~(\ref{qfree}).

The first rigorous mathematical work on the directed polymers in $1+d$ dimensions was done by Imbrie and Spencer \cite{IS}, proving that in dimension $d\geq 3$ with Bernoulli disorder and small enough $\beta,$ the end point of the polymer scales as $n^{1/2},$ i.e. the polymer is diffusive. Later, Bolthausen \cite{Bolt} extended this to a central
 limit theorem for the end point of the walk, showing that the polymer behaves almost as if the disorder were absent. In the same paper, Bolthausen also introduced the nonnegative martingale $M_n(\beta)=Z_n(\beta)/\mathbf{E}[Z_n(\beta)]$ and observed that for the positivity of the limit $M(\beta)=\lim_{n\to \infty}M_n(\beta)$, there are only two possibilities, $\mathbf{P}(M(\beta)>0)=1$, known as \emph{weak disorder}, or $\mathbf{P}(M(\beta)=0)=1$, known as \emph{strong disorder}.
Comets and Yoshida \cite{CSY1, CY}, showed that there exists a critical value $\beta_{c}\in [0,\infty]$, with 
$\beta_c =0 $ for $d=1,2$ and  $0<\beta_c \leq \infty$ for $ d \geq 3$, such that 
$\mathbf{P}(M(\beta)>0)=1$  if $ \beta \in \{0\}\cup(0,\beta_c)$ and $\mathbf{P}(M(\beta)=0)=1$  if $ \beta >\beta_c.$ In particular, for the $1+1$ dimensional case, disorder is always strong. It is not known whether $\beta_c$ belongs to the weak disorder or strong disorder phase for the lattice version, whereas we know that $\beta_c$ belongs to the strong disorder phase for the tree case.

In the $1+1$ dimensional case, the \textit{localized microscopic defect} is incorporated to the model by modifying the Hamiltonian as follows:

\begin{equation}
  H^u_n(S):=\sum_{j=1}^{n}(v(j, S_j)+u1_{S_j=0})
\end{equation}
and $Z_n(\beta,u):=\sum_{S}e^{\beta H^u_n(S)}$
denotes the partition function. The free energy and the critical curve of the model are defined as
\begin{equation}
\label{Lqfree}
f(\beta,u):=\lim_{n\to \infty}\frac{1}{n}\log Z_n(\beta,u)
\end{equation}
\[u_c(\beta):=\inf\{u\geq0: f(\beta,u)>f(\beta,0)\}.\]
For the existence of the limit in Equation~(\ref{Lqfree}) and its self-averaging property, see~\cite{AlYil}.
As we discussed in section~\ref{bulkvsmd}, the question of whether $u_c(\beta)>0$ or not for some range of $\beta$ is still an open question. One of the reasons why it is not easy to solve this question rigorously is that a nice decomposition, such as in Equation~(\ref{eq.ZsumG1}), is not available for the partition function of the lattice model.

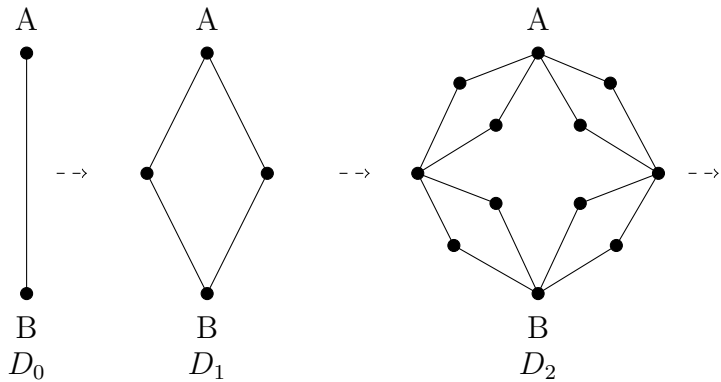
\begin{figure}
\centering
\begin{tikzpicture}[scale=0.8]
\draw (0,0) -- (0,4);
\draw [fill] (0,0) circle [radius=0.1];
\node [below] at (0,-0.2) {B};
\node [below] at (0,-0.8) {$D_0$};
\draw [fill] (0,4) circle [radius=0.1];
\node [above] at (0,4.2) {A};
\draw[->, dashed] (0.5,2)--(1,2);
\draw[->, dashed] (5.2,2)--(5.7,2);
\draw[->, dashed] (11,2)--(11.5,2);
\draw (2,2) -- (3,4);
\draw (3,0) -- (2,2);
\draw (4,2) -- (3,4);
\draw (3,0) -- (4,2);
\draw [fill] (2,2) circle [radius=0.1];
\draw [fill] (4,2) circle [radius=0.1];
\draw [fill] (3,0) circle [radius=0.1];
\node [below] at (3,-0.2) {B};
\node [below] at (3,-0.8) {$D_1$};
\draw [fill] (3,4) circle [radius=0.1];
\node [above] at (3,4.2) {A};
\draw [fill] (8.5,0) circle [radius=0.1];
\node [below] at (8.5,-0.2) {B};
\node [below] at (8.5,-0.8) {$D_2$};
\draw [fill] (8.5,4) circle [radius=0.1];
\node [above] at (8.5,4.2) {A};
\draw [fill] (6.5,2) circle [radius=0.1];
\draw [fill] (10.5,2) circle [radius=0.1];
\draw [fill] (7.2,3.5) circle [radius=0.1];
\draw [fill] (7.8,2.8) circle [radius=0.1];
\draw (6.5,2)--(7.2,3.5);
\draw (6.5,2)--(7.8,2.8);
\draw (7.2,3.5)--(8.5,4);
\draw (7.8,2.8)--(8.5,4);
\draw [fill] (9.7,3.5) circle [radius=0.1];
\draw [fill] (9.2,2.8) circle [radius=0.1];
\draw (10.5,2)--(9.7,3.5);
\draw (10.5,2)--(9.2,2.8);
\draw (9.7,3.5)--(8.5,4);
\draw (9.2,2.8)--(8.5,4);
\draw (6.5,2)--(7.1,0.8);
\draw (6.5,2)--(7.8,1.5);
\draw (7.1,0.8)--(8.5,0);
\draw (7.8,1.5)--(8.5,0);
\draw [fill] (7.1,0.8) circle [radius=0.1];
\draw [fill] (7.8,1.5) circle [radius=0.1];
\draw (10.5,2)--(9.8,0.8);
\draw (10.5,2)--(9.2,1.5);
\draw (9.8,0.8)--(8.5,0);
\draw (9.2,1.5)--(8.5,0);
\draw [fill] (9.8,0.8) circle [radius=0.1];
\draw [fill] (9.2,1.5) circle [radius=0.1];
\end{tikzpicture}
\caption{The recursive construction of the first three generations of the diamond lattice $D_n$, a special case of the hierarchical lattice for $b=2, s=2$. Each site of the lattice carries a random disorder, and directed paths from A to B represent the polymer configurations.} 
\label{fig-hierar}
\end{figure}
\begin{figure}
\centering
\begin{tikzpicture}[scale=0.8]
\draw [fill] (8.5,0) circle [radius=0.1];
\node [below] at (8.5,-0.2) {B};
\draw [fill] (8.5,4) circle [radius=0.1];
\node [above] at (8.5,4.2) {A};
\draw [fill] (6.5,2) circle [radius=0.1];
\draw [fill] (10.5,2) circle [radius=0.1];
\draw [fill] (7.2,3.5) circle [radius=0.1];
\draw [fill] (7.8,2.8) circle [radius=0.1];
\draw[ultra thick] (6.5,2)--(7.2,3.5);
\draw (6.5,2)--(7.8,2.8);
\draw [ultra thick](7.2,3.5)--(8.5,4);
\draw (7.8,2.8)--(8.5,4);
\draw [fill] (9.7,3.5) circle [radius=0.1];
\draw [fill] (9.2,2.8) circle [radius=0.1];
\draw (10.5,2)--(9.7,3.5);
\draw (10.5,2)--(9.2,2.8);
\draw (9.7,3.5)--(8.5,4);
\draw (9.2,2.8)--(8.5,4);
\draw [ultra thick](6.5,2)--(7.1,0.8);
\draw (6.5,2)--(7.8,1.5);
\draw [ultra thick](7.1,0.8)--(8.5,0);
\draw (7.8,1.5)--(8.5,0);
\draw [fill] (7.1,0.8) circle [radius=0.1];
\draw [fill] (7.8,1.5) circle [radius=0.1];
\draw (10.5,2)--(9.8,0.8);
\draw (10.5,2)--(9.2,1.5);
\draw (9.8,0.8)--(8.5,0);
\draw (9.2,1.5)--(8.5,0);
\draw [fill] (9.8,0.8) circle [radius=0.1];
\draw [fill] (9.2,1.5) circle [radius=0.1];
\end{tikzpicture}
\caption{Hierarchical lattice with a defect branch. The thick bonds represent the defect branch. The disordered variables along the defect branch is enhanced by a fixed potential $u$. The polymer will follow the defect branch for $u>u_c(\beta,b,s)$ depending on the inverse temperature $\beta$ and the lattice parameters $b$ and $s$.}
\label{fig-hierardef}
\end{figure}
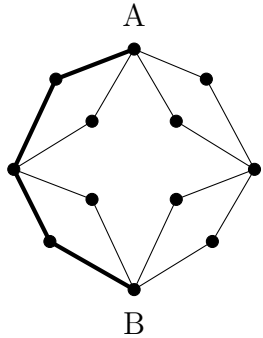
\subsection{Directed polymers on disordered hierarchical lattices with defect substructure} 
The directed polymers on disordered hierarchical lattices were first introduced and studied in the 
physics literature by Derrida and Griffiths \cite{DerGr}, and Cook and Derrida \cite{CoDer} for the \textit{bond} disordered case, and then rigorously by Lacoin and Moreno \cite{LaMo} for the \textit{site} disordered case. The hierarchical lattices are usually generated by an iterative rule as described for the \textit{diamond} lattice: The first generation, $D_0$, consists of two sites, labeled as $A$ and $B$, with one bond. In the next generation, $D_1$, the bond is replaced by a set of four bonds, and then in each step, each bond is replaced by such a set of four bonds to form the next generation, see Figure~\ref{fig-hierar}.
For more general hierarchical lattices, the generation $D_{n+1}$ is obtained by replacing each bond in the generation $D_n$ by $b$ branches of $s$ bonds. The directed paths in $D_n$ linking the sites $A$ and $B$ represent the polymer configurations. The disorder is introduced in the model by assigning independent random variables from a distribution to each site. The Hamiltonian of the model, partition function, and free energy are defined as in lattice and tree version of the model, and the martingale defined by the normalized partition function separates two phases as \textit{weak} and \textit{strong} disorder depending on the lattice parameters $b, s$ and the inverse temperature $\beta$, for the details see \cite{LaMo}. In \cite{LaMo}, in particular, they prove that the free energy exists almost surely and it is a strictly convex function of $\beta$ which holds also for the directed polymer on $\mathbb{Z}^d$ but not on the tree for $\beta>\beta_c$. As noted in \cite{LaMo}, this  fact is related to the ``correlation structure'' of the models as two directed paths on $\mathbb{Z}^d$ and hierarchical lattice can re-intersect after being separated at some point which is not the case for the tree model. Among these three directed polymer models, only the tree version is exactly solvable, that is, there is an explicit expression for the free energy. 

The localized microscopic defect is incorporated to the model by enhancing the disorder variables along a single directed path from A to B with a fixed potential $u$, see Figure~\ref{fig-hierardef}. The main question is determining whether the critical point for the extra potential is zero or not depending on the model parameters, inverse temperature $\beta$, and lattice parameters $b,s$; that is whether $u_c(\beta,b,s)=0$ or not for some $\beta, b, s$. This problem was studied in \cite{BK1} by using Migdal-Kadanoff renormalization group method but the results lack the rigor of formal mathematical proofs.

\ack
N. Madras was supported in part by a Discovery Grant from NSERC of Canada.


\section*{References}
\begin{thebibliography}{99}

\bibitem{AlYil} Alexander, K. S. and  Y\i ld\i r\i m, G. \newblock Directed polymers in a random environment with a defect line. \newblock{\em Electronic Journal of Probability}, 20, no 6, 20 pp, 2015.

\bibitem{BK}
Balents, L. and Kardar, M. \newblock Delocalization of flux lines from extended
  defects by bulk randomness. \newblock{\em Europhys. Lett.}, 23, 503--509, 1993.


\bibitem{BK1}
Balents, L. and Kardar, M. \newblock Disorder-induced unbinding of a fiux line from an extended defect. \newblock{\em Physical Review B}, 49, 18, 13030(19), 1994.

\bibitem{BSS}
Basu, R., Sidoravicius, V. and Sly, A. \newblock Last passage percolation with a defect
  line and the solution of the slow bond problem. \newblock{\em https://arxiv.org/abs/1408.3464}. 2014
  
\bibitem{BBDD} Beaton, N., Bousquet-M\'{e}lou, M., De Gier, J., Duminil-Copin, H. and Guttmann, A. \newblock The critical fugacity for surface adsorption of self-avoiding walks on the honeycomb lattice is $1+\sqrt 2$. \newblock{\em Communications in Mathematical Physics}, 326(3), 727--754, 2014.

\bibitem{BSSV}
Beffara, V., Sidoravicius, V., Spohn, H. and Vares, M. E. \newblock Polymer pinning in a random medium as influence percolation.\newblock{\em Dynamics and Stochastics, IMS Lecture Notes Monogr. Ser.} 48, Inst. Math. Statist., Beachwood, OH, 2006. 

\bibitem{BSV}
Beffara, V., Sidoravicius, V. and Vares, M.~E. \newblock Randomized polynuclear growth with a columnar defect. \newblock{\em Probab. Theory Rel. Fields}, 147,
  no.~3-4, 565--581, 2010.
  
\bibitem{Biggins} Biggins, J. D. \newblock Martingale convergence in the branching random walk. \newblock{\em J. Appl. Probability}, 14, no. 1, 25--37, 1977.

\bibitem{Bolt} Bolthausen, E. \newblock A note on the diffusion of directed polymers in a random environment. \newblock{\em Comm. Math. Phys.}, 123, no. 4, 529--534, 1989.

\bibitem{BSL}
Budhani, R.~C., Swenaga, M. and Liou, S.~H.\newblock Giant suppression of flux-flow
  resistivity in heavy-ion irradiated ${\rm TL_2BA_2Ca_2Cu_3O_{10}}$ films:
  Influence of linear defects on vortex transport. \newblock{\em Phys. Rev. Lett.}, 69, 3816--3819, 1992.
  
\bibitem{BPP} Buffett, E., Patrick, A. and Pul\'{e}, J. V. \newblock Directed polymers on trees: a martingale approach. \newblock{\em J. Phys. A: Math. Gen.}, 26, 1823--1834, 1993.  

\bibitem{CaHu1} Carmona, P. and Hu, Y. \newblock On the partition function of a directed polymer in a random environment. \newblock{\em Probab. Theory Rel. Fields}, 124, 431--457, 2002. 

\bibitem{CMWT}
Civale, L., Marwick, A.~D., Worthington, T.~K., Kirk, M.~A., Thompson, J.~R., Krusin-Elbaum, L., Sum, Y., Clem, Y. R. and Holtzberg, F. \newblock Vortex
  confinement by columnar defects in ${\rm YBa_2Cu_3O_7}$ crystals: Enhanced pinning at
  high fields and temperatures. \newblock{\em Phys. Rev. Lett.}, 67, 648--652, 1991.


\bibitem{Comets} Comets, F. \newblock Directed Polymers in Random Environment. \newblock{\em Lecture notes for a workshop on Random Interfaces and Directed Polymers}, Leipzig, 2005.


\bibitem{Comets1}
  Comets, F.
  \newblock Directed polymers in random environments. \newblock{\em Lecture notes from the $46^{th}$ Probability Summer School held in Saint--Flour}, Springer--Verlag, 2016.
  
\bibitem{CSY1} Comets, F., Shiga, T. and Yoshida, N. \newblock Directed polymers in a random environment: path localization and strong disorder. \newblock{\emph Bernoulli}, 9, no. 4, 705--723, 2003.

\bibitem{CY} Comets, F. and Yoshida, N. \newblock Directed polymers in a random environment are diffusive at weak disorder. \newblock{\emph Ann. Probab.,} 34, (5) 1746--1770, 2006.

\bibitem{CoDer} Cook, J. and Derrida, B. \newblock Polymers on Disordered Hierarchical Lattices: A Nonlinear Combination
of Random Variables. \newblock{\em J. Stat. Phys.}, 57, 1/2, 89--139, 1989.

\bibitem{denH} den Hollander, F. \newblock Random polymers. \newblock{\em Lectures from the $37^{th}$ Probability Summer School held in Saint-Flour}. Springer--Verlag, 2009.

\bibitem{DerSp} Derrida, B. and Spohn, H. \newblock Polymers on disordered trees, spin glasses, and traveling waves. \newblock{\em J. Statist. Phys.}, 51, no. 5–6, 817--840, 1988.


\bibitem{DerGr} Derrida, B. and Griffiths, R. B. \newblock{Directed polymers on disordered hierarchical lattices.} \newblock{\em Europhys.
Lett.}, 8, 2, 111--116, 1989.



\bibitem{Durrett} Durrett, R. \newblock Probability: Theory and Examples. $4^{th}$ edition.  \newblock{\em Cambridge Series in Statistical and Probabilistic Mathematics, 31. Cambridge University Press}, Cambridge, 2010.

\bibitem{Giacomin} Giacomin, G. \newblock Disorder and critical phenomena through basic probability models. \newblock{\em \'Ecole d'\'Et\'e de Probabilit\'es de Saint-Flour XL, Lecture Notes in Mathematics}, Springer, 2010. 

 \bibitem{Giacomin1} Giacomin, G. \newblock Random polymer models. \newblock{\em Imperial College Press, World Scientific.}  2007.
 


\bibitem{Hamm1} Hammersley, J. M. \newblock Critical phenomena in semi-infinite systems.
\newblock In Essays in Statistical Science.
\newblock{\em J. Appl.\ Prob.}, 19A, 327--331, 1982.

 \bibitem{Hamm2} Hammersley, J. M., Torrie, G. M., and Whittington, S. G. 
 \newblock Self-avoiding walks interacting with a surface.
\newblock{\em J. Phys.\ A: Math.\ Gen.}, 15, 539--571, 1982.

 
\bibitem{HHR}
Hansen, A., Hinrichsen,  E.~L. and Roux, S. \newblock Roughness of crack interfaces.
  \newblock{\em Phys. Rev. Lett.}, 66, 2476--2479, 1991.


\bibitem{HuShi} Hu, Y. and Shi, Z. \newblock Minimal position and critical martingale convergence in branching random walks, and directed polymers on disordered trees. \newblock{\em Ann. Probab.,} 37, no. 2, 742--789, 2009.

\bibitem{HH}
Huse, D.~A. and Henley, C.~L. \newblock Pinning and roughening of domain wall in
  Ising systems due to random impurities. \newblock{\em Phys. Rev. Lett.}, 54,
  2708--2711, 1985.

\bibitem{HN}
Hwa, T. and Nattermann, T. \newblock Disorder-induced depinning transition. \newblock{\em Phys.
  Rev. B}, 51, no.~1, 455--469, 1995.
  
\bibitem{IS}
Imbrie, J. Z., Spencer, T. \newblock Diffusion of directed polymers in a random
  environment. \newblock{\em J. Statist. Phys.}, 52, no.~3-4, 609--626, 1988.
  
\bibitem{KahPey} Kahane, J. P.  and Peyriére, J. \newblock Sur certaines martingales de Benoit Mandelbrot. \newblock{\em Advances in Math.}, 22, no. 2, 131--145, 1976.  

\bibitem{K2}
Kardar, M. \newblock Depinning by quenched randomness. \newblock{\em Phys. Rev. Lett.}, 55, 2235--2238, 1985.

\bibitem{KPZ}
Kardar, M., Parisi, G. and Zhang, Y. C. \newblock Dynamic scaling of growing
  interfaces. \newblock{\em Phys. Rev. Lett.}, 56, no.~9, 889--892, 1986.



\bibitem{LaMo} Lacoin, H. and Moreno, G. \newblock{Directed polymers on hierarchical lattices with site disorder.} \newblock{\em Stochastic Processes and their Applications.}, 120, 467--493, 2010. 

 \bibitem{Madras} Madras, N. 
 \newblock Location of the adsorption transition for lattice polymers.
\newblock{\em J. Phys.\ A: Math.\ Theor.}, 50, 064003, 2017.



  
\bibitem{MO} Morters, P. and Ortgiese, M. \newblock Minimal supporting subtrees for the free energy of polymers on disordered trees. \newblock{\em J. Math. Phys.} 49, no. 12, 125203, 21 pp, 2008. 


\bibitem{Nel}
Nelson, D.~R. \newblock Vortex entanglement in high-${\rm t_c}$ superconductors.
  \newblock{\em Phys. Rev. Lett.}, 60, 1973--1976, 1988.

 \bibitem{Rubin} Rubin, R. J.
 \newblock Random-walk model of chain polymer adsorption at a surface.
\newblock{\em J. Chem.\ Phys.}, 43, 2392--2407, 1965.



\bibitem{TL}
Tang, L.~H. and Lyuksyutov, I.~F. \newblock Directed polymer localization in a
  disordered medium. \newblock{\em Phys. Rev. Lett.}, 71, 2745--2748, 1993.

\end{thebibliography}
\end{document}